\documentclass[12pt]{amsart}
\usepackage{amsmath,amssymb,cite}

\def\rn{\mathbb R^n}

\def\N{\mathbb N}
\def\Z{\mathbb Z}
\def\K{\mathbb K}

\def\M{\mathfrak M}
\def\vp{\varphi}

\def\ls{\lesssim}

\newcommand{\lt}{\left(}
\newcommand{\rt}{\right)}
\newcommand{\ds}{\mathrm{\,d}s}
\newcommand{\dt}{\mathrm{\,d}t}

\newcommand{\dy}{\mathrm{\,d}y}
\newcommand{\dz}{\mathrm{\,d}z}
\newcommand{\dtau}{\mathrm{\,d}\tau}

\newcommand{\esssup}{\operatornamewithlimits{ess\, sup\,}}

\newcommand{\supp}{\operatorname{supp}}
\newcommand{\MM}{\mathfrak M_+}

\theoremstyle{plain}
\newtheorem{theorem}{Theorem}[section]
\newtheorem{lemma}[theorem]{Lemma}
\newtheorem{corollary}[theorem]{Corollary}
\newtheorem{proposition}[theorem]{Proposition}

\theoremstyle{definition}

\newcounter{bcount}
\newcommand{\Bdef}[1]{\refstepcounter{bcount}\label{bcounter#1}}
\newcommand{\B}[1]{B_{\ref{bcounter#1}}}

\newtoks\by
\newtoks\paper
\newtoks\book
\newtoks\jour
\newtoks\yr
\newtoks\pages
\newtoks\vol
\newtoks\publ
\newtoks\eds
\newtoks\proc
\def\ota{{\hbox{???}}}
\def\cLear{\by=\ota\paper=\ota\book=\ota\jour=\ota\yr=\ota
\pages=\ota\vol=\ota\publ=\ota}
\def\endpaper{\the\by, \textit{\the\paper},
{\the\jour} {\the\vol} (\the\yr), \the\pages.\cLear}
\def\endbook{\the\by, \textit{\the\book}, \the\publ.\cLear}
\def\endprep{\the\by, \textit{\the\paper}, \the\jour.\cLear}
\def\endproc{\the\by, \textit{\the\paper}, \the\publ, \the\pages.\cLear}
\def\name#1#2{#1 #2}
\def\et{ and }

\begin{document}

\title[Weighted inequalities for iterated Copson operators]
{Weighted inequalities for iterated Copson integral operators}
\author {Martin K\v repela and Lubo\v s Pick}

\address{Martin K\v repela,
Institute of Mathematics, University of
Freiburg,  Ernst-Zermelo-Stra\ss e 1, 791 04 Freiburg, Germany}
\email{martin.krepela@math.uni-freiburg.de}

\address{Lubo\v s Pick, Department of Mathematical Analysis\\
Faculty of Mathematics and Physics\\
Charles University\\
Sokolovsk\'a~83\\
186~75 Praha~8\\
Czech Republic}
\email{pick@karlin.mff.cuni.cz}

\subjclass[2010]{47G10, 47B38, 26D10}
\keywords{weighted Copson inequalities; discretization; integral operators; iterated operators.}

\begin{abstract}
We solve a long-standing open problem in theory of weighted inequalities concerning iterated Copson operators. We use a constructive approximation method based on a new discretization principle that is developed here. In result, we characterize all weight functions $w,v,u$ on $(0,\infty)$ for which there exists a constant $C$ such that the inequality
\[
\left(\int_0^{\infty}\!\!\left(\int_t^\infty \!\!\left(\int_s^{\infty}\!\!h(y)\dy\right)^{\!m}\!\!u(s) \ds\right)^{\!\frac{q}{m}}\!\!w(t)\dt\right)^{\!\frac{1}{q}}
\!\leq C
\left(\int_0^{\infty}\!\!h(t)^pv(t)\dt\right)^{\!\frac{1}{p}}
\]
holds for every non-negative measurable function $h$ on $(0,\infty)$, where $p,q$ and $m$ are positive parameters. We assume that $p\geq 1$ because otherwise the studied inequality cannot hold for nontrivial weights, but otherwise $p,q$ and $m$ are unrestricted.
\end{abstract}

\maketitle

\section{Introduction and the main result}

In the theory of weighted inequalities, a notorious difficult problem which has been open for many years is that of characterization of inequalities involving the Copson integral $\int_t^{\infty}h(s)\ds$. A~specific and particularly important variant of these concerns \textit{iterated} Copson integral. This type of a~problem has been surfacing for some time in various research fields, quite different in nature. The inequality involving the iterated Copson operator is traditionally considered as a very difficult one to handle, as it contains three independent weights and three possibly different power parameters. Solving this problem is the aim of this paper.\\

More precisely stated, we will concentrate on the inequality
\begin{equation}\label{E:main}
\left(\int_0^{\infty}\hspace{-5pt}\left(\int_t^\infty \hspace{-5pt}\left(\int_s^{\infty}\hspace{-5pt}h(y)\dy\right)^m \hspace{-5pt} u(s) \ds\right)^{\frac{q}{m}} \hspace{-5pt} w(t)\dt\right)^{\frac{1}{q}}
\hspace{-5pt}\le C
\left(\int_0^{\infty}\hspace{-5pt}h(t)^pv(t)\dt\right)^{\frac{1}{p}}\hspace{-5pt},
\end{equation}
in which $h$ is a non-negative measurable function on $(0,\infty)$ and $C$ is independent of $h$. The task is to characterize those parameters and weights for which a~constant $C$, independent of the~positive function $h$, can be found so that~\eqref{E:main} holds. Our approach to this problem involves a~mixture of discretization and anti-discretization techniques combined with certain duality principles with respect to the integral pairing. The restriction $p\geq 1$ is necessary because otherwise the inequality~\eqref{E:main} cannot hold for all $h\geq0$ unless at least one of the weights  $u,w$ is trivial. To see this, it suffices to construct a function $h$ such that the quantity on the right is finite but which is not locally integrable (which is easy).\\

There is a significant motivation for this problem to be cracked. The inequalities of the form~\eqref{E:main} govern surprisingly many important principles in analysis. One can, for example, recall that the search for an~explicit formula for the optimal rearrangement-invariant function norm in a Sobolev inequality (which is an indispensable tool in the regularity theory of solutions to degenerate elliptic partial differential equations) -- see~\cite[Theorem~3.1]{T5} and the references therein. It turns out that the problem is equivalent to a~balance condition which can be transformed into an inequality of the form~\eqref{E:main}. In a different world, see for instance the recent work~\cite{ACS}, weighted inequalities involving Copson integrals are used to investigate the continuity properties of local solutions to the $p$-Laplace equation
\[
-\operatorname{div}(|\nabla u|^{p-2}\nabla u)=f(x)\quad \textup{in}\ \Omega,
\]
where $\Omega$ is an~appropriate bounded open subset of $\rn$. Certain theoretical parts of functional analysis, such as the theory of function spaces or interpolation theory more often than not produce questions that lead to inequalities involving Copson integrals. For example, in the study of embeddings of Lorentz-type structures involving two possibly different weighted integral means (which itself has plenty of underlying applications), see~\cite{kostlivec} and its references, one makes a~direct use of such inequalities.\\

There are several angles from which the inequality~\eqref{E:main} can be viewed. For example, one may notice that the expression $\int_t\sp{\infty}h(s)\ds$ represents a~nonincreasing function on $(0,\infty)$ when $h$ is non-negative. This may enable us to view~\eqref{E:main} as an inequality working on the cone of monotone functions. Inequalities involving such functions on an interval have been seeing a~constant boom ever since 1990 and a vast amount of literature is nowadays available, see e.g.~\cite{AM, CGS,CPSS,CS1,CS2,CS3,GP2,G1,G2,GHS,HM,La,Neu,Oi,Saw,Si,SS,Sor,St,St2}. Another possible point of view is to consider the inequality~\eqref{E:main} as a~certain embedding relation between two structures, possibly function spaces. This observation leads to a~possible broadening of techniques that might come useful, in particular, e.g.~the duality principles and reduction theorems.\\

The use of discretization techniques has an interesting history which seems to be far from being finished. The classic book of this discipline is~\cite{KGGE}. A~significant breakthrough was brought to the field with the discovery of a~new approach based on discretization and anti-discretization techniques in \cite{GP}. Using these new thoughts, embeddings of classical Lorentz spaces were characterized in cases that had before resisted for years, and (very useful) two-operator inequalities could be treated -- see~\cite{CGMP1} and~\cite{CGMP2}. Applications to duality theorems and reduction principles naturally followed -- see e.g.~\cite{GP4,GP-nrt}.\\

Interestingly, the discretization techniques were at one stage considered too technical by several authors, and some nontrivial effort was spent in order to recover the results by methods that would avoid discretization. This, however, was met merely with a~little success (see e.g.~\cite{GS1,GS2,GK}), and the approximation techniques stood their ground.  It should be mentioned that there is a~significant demand for obtaining two-sided estimates of the optimal (that is, smallest possible) constant in~\eqref{E:main}. This often arises from tasks that involve, in some way, two different integral mean operators, and it appears, for example, in results concerning the optimality of function spaces in Sobolev-type embeddings, carried out e.g.~in~\cite{EKP,T2,T3,T5,CP-gauss}. Yet another reason for the interest in the inequalities of the type~\eqref{E:main} is their connection to the grand Lebesgue spaces and the small Lebesgue spaces. These spaces were introduced in~\cite{IS} and~\cite{Fio} in connection with integrability of Jacobians and in a~direct connection to problems in mechanics and mathematical physics concerning, for instance, mappings of a~finite distortion. The connection to weighted inequalities was pointed out several times, see e.g.~\cite{FK}. Duality techniques related to this field were further polished in~\cite{GPS} and~\cite{kostlivec}.\\

In connection with the problem in hand, we refine certain ideas from~\cite{GP} in the spirit of~\cite[Theorem~3.8]{copson} in order to fit the specific needs of iterated Copson operators.
In result, we will provide a~full characterization of weights and parameters for which the inequality~\eqref{E:main} holds and give two-sided estimates for its optimal constant.\\

The inequality~\eqref{E:main} was also treated in~\cite{GM,Mus} by a~completely different approach. The criteria obtained in those papers are of different nature than those established here. In~\cite{PS1} iterated operators involving kernels are studied by different methods. The iteration of Copson and Hardy operators constitutes a related but in fact an essentially different problem. Inequalities involving such operators have been investigated for instance in~\cite{GMP1,GMP2,Pro,PS2}.\\

We are about to state the results. Let us first briefly introduce the necessary notation. By $\MM$ we denote the cone of all non-negative measurable functions on $(0,\infty)$. Should a~different underlying interval $(a,b)$ be needed, then we would write $\MM(a,b)$. If $q\in (0,1)\cup(1,\infty)$, then the symbol $q'$ denotes the conjugate exponent $q'=\frac{q}{q-1}$. Notice that $q'$ is negative if $q<1$. The same notation is used for $p$ and $p'$. The notation $A\lesssim B$ means that there exists a~constant $c\in(0,\infty)$ ``independent of relevant quantities in $A$ and $B$'' and such that $A\le cB$. Precisely, such $c$ may depend only on the exponents $m$, $p$, $q$. We write $A\approx B$ if both $A\lesssim B$ and $B\lesssim A$ are true.
Throughout the paper, the conventions
\[
0\cdot\infty=0, \qquad \frac{0}{0}=0, \qquad a^0 = 1 \text{ for all } a\in[0,\infty]
\]
are in effect to avoid undefined expressions.\\

By a~\textit{weight} we mean a~measurable non-negative function on $(0,\infty)$. Let $u$ and $w$ be weights such that the expression $\vp(t)$, defined for any $t\in[0,\infty]$ by
\[
\vp(t)=
\left(\int_0^{t}\left(\int_s^t u(y) \dy\right)^{\frac{q}{m}}w(s)\ds\right)^{\frac{1}{q}},
\]
satisfies $0<\vp(t)<\infty$ for every $t\in(0,\infty)$. In this case, we say that the pair $(u,w)$ is~\textit{admissible} with respect to $(m,q)$. We note that $\vp$ is a~nondecreasing continuous function on $(0,\infty)$ satisfying $\lim_{t\to0_+}\vp(t)=0$, and that the derivative $\vp'$ exists a.e.~on $(0,\infty)$.\\

Now we state our main result. For convenience it is split into two theorems, a~proposition, and a~corollary.

\begin{theorem}\label{T:main}
Let $p\in(1,\infty)$ and $q,m\in(0,\infty)$. Let $u,w,v$ be weights such that the pair $(u,w)$ is admissible with respect to $(m,q)$. Let
	\begin{equation}\label{E:C}
C=\sup_{h\in\M_+}
\frac
{\left(\int_0^{\infty}\left(\int_t^\infty \left(\int_s^{\infty}h(y)\dy\right)^mu(s) \ds\right)^{\frac{q}{m}}w(t)\dt\right)^{\frac{1}{q}}}
{\left(\int_0^{\infty}h(t)^pv(t)\dt\right)^{\frac{1}{p}}}.
	\end{equation}

\textup{(a)} Let $1<p\leq m<\infty$ and $p\leq q<\infty$. Then $C\approx A_1$, where
\[
A_1=
\sup_{t\in(0,\infty)}\left(\int_0^tw(s)\left(\int_s^t u(y)\dy\right)^{\frac qm}\ds\right)^{\frac1q}
\left(\int_t^{\infty}v(s)^{1-p'}\ds\right)^{\frac1{p'}}.
\]

\textup{(b)} Let $1<p\leq m<\infty$ and $0<q<p$. Define $r= \frac{pq}{p-q}$. Then $C\approx A_2+A_3$, where
\[
A_2=
\left(\int_0^{\infty}\hspace{-5pt}
\left(\int_0^t w(s)\ds\right)^{\frac rp} \hspace{-2pt} w(t) \hspace{-2pt}
\sup_{z\in(t,\infty)} \hspace{-2pt}
\left(\int_t^z \hspace{-2pt}u(y)\dy\right)^{\hspace{-2pt}\frac rm}\hspace{-5pt}
\left(\int_z^{\infty}\hspace{-5pt}v(y)^{1-p'}\dy\right)^{\frac r{p'}}\hspace{-5pt}\dt\right)^{\frac 1r}
\]
and
\begin{align*}
A_3=&
\left(\int_0^{\infty}
\left(\int_0^tw(s)\left(\int_s^t u(y)\dy\right)^{\frac qm}\ds\right)^{\frac rp}
 w(t) \right.\\
&\qquad \times \left.\sup_{z\in(t,\infty)}
\left(\int_t^z u(y)\dy\right)^{\frac qm}
\left(\int_z^{\infty}v(y)^{1-p'}\dy\right)^{\frac r{p'}}\dt\right)^{\frac 1r}.
\end{align*}

\textup{(c)} Let $0<m<p<\infty$ and $1<p\leq q$. Then $C\approx A_1+A_4\approx A_1+A_4^*$, where
\begin{align*}
A_4&=
\sup_{t\in(0,\infty)}\left(\int_0^tw(s)\ds\right)^{\frac1q} \\
& \qquad \qquad \times\left(\int_t^{\infty}\hspace{-5pt}\left(\int_t^s u(y)\dy\right)^{\!\frac p{p-m}}\hspace{-5pt}\left(\int_{s}^{\infty}\hspace{-5pt}v(\tau)^{1-p'}\dtau\right)^{\hspace{-5pt}\frac{p(m-1)}{p-m}}\hspace{-8pt}v(s)^{1-p'}\!\ds\right)^{\!\!\frac{p-m}{pm}}
\end{align*}
and
\begin{align*}
A_4^*&=
\sup_{t\in(0,\infty)}\left(\int_0^tw(s)\ds\right)^{\frac1q} \\
&\qquad\qquad \times\left(\int_t^{\infty}\hspace{-5pt}\left(\int_t^s u(y)\dy\right)^{\hspace{-2pt}\frac m{p-m}}u(s)\left(\int_{s}^{\infty}\hspace{-5pt}v(\tau)^{1-p'}\dtau\right)^{\hspace{-5pt}\frac{m(p-1)}{p-m}}\ds\right)^{\hspace{-5pt}\frac{p-m}{pm}}\hspace{-5pt}.
\end{align*}
If, moreover, $m\geq1$ or $\int_{t}^{\infty}v(s)^{1-p'}\ds<\infty$ for every $t\in(0,\infty)$, then $A_4\approx A_4^*$.

\textup{(d)} Let $0<m<p<\infty$, $1<p$ and $0<q<p<\infty$. Define $r= \frac{pq}{p-q}$. Then $C\approx A_3+A_5\approx A_3+A_5^*$, where
\begin{align*}
A_5&=
\left(\vphantom{\Bigg)^\frac12} \int_0^{\infty}
\left(\int_0^tw(s)\ds\right)^{\frac rp}w(t) \right.\\
&\quad\times\left.
\left(\int_t^{\infty}\hspace{-5pt}\left(\int_t^s \hspace{-2pt} u(y)\dy\right)^{\hspace{-2pt}\frac p{p-m}}\hspace{-5pt}\left(\int_{s}^{\infty}\hspace{-5pt}v(\tau)^{1-p'}\dtau\right)^{\hspace{-5pt}\frac{p(m-1)}{p-m}}\hspace{-8pt}v(s)^{1-p'}\ds\right)^{\hspace{-5pt}\frac{q(p-m)}{m(p-q)}}\hspace{-7pt}\dt
\right)^{\frac 1r}
\end{align*}
and
\begin{align*}
A_5^*&=
\left(\vphantom{\Bigg)^\frac12}\int_0^{\infty}
\left(\int_0^tw(s)\ds\right)^{\frac rp}w(t) \right.\\
&\quad\times\left.
\left(\int_t^{\infty}\hspace{-5pt}\left(\int_t^s \hspace{-2pt}u(y)\dy\right)^{\!\!\frac m{p-m}} \hspace{-5pt} u(s)\left(\int_{s}^{\infty} \hspace{-5pt} v(\tau)^{1-p'}\dtau\right)^{\hspace{-5pt}\frac{m(p-1)}{p-m}}\hspace{-5pt}\ds\right)^{\frac{q(p-m)}{m(p-q)}}\hspace{-5pt}\dt
\right)^{\frac 1r}.
\end{align*}
If, moreover, $m\geq1$ and $q>1$ or $\int_{t}^{\infty}v(s)^{1-p'}\ds<\infty$ for every $t\in(0,\infty)$, then $A_5\approx A_5^*$.
\end{theorem}

In the case (d) of the previous theorem, a~simpler alternative equivalent condition is available if we are willing to settle for the restriction $m\leq q$. Since this is often the case in applications, this fact is worth pointing out.

\begin{proposition}\label{T:dodatek}
Let $0<m<p<\infty$, $1<p$ and $0<q<p<\infty$. Define $r= \frac{pq}{p-q}$. If $q=1$, define
$\frac r{q'}=0$. Let $u,w,v$ be weights such that the pair $(u,w)$ is admissible with respect to $(m,q)$. Let $C$, $A_1$, $A_5$ be defined as in Theorem \ref{T:main}, and let
\[
A_6=
\left(\int_0^{\infty}\hspace{-5pt}
\left(\int_0^tw(s)
\left(\int_s^{t}u(y)\dy\right)^{\hspace{-2pt}\frac q{m}}\hspace{-5pt}\ds\right)^{\hspace{-2pt}\frac rq} \hspace{-5pt}
\left(\int_t^{\infty}v(y)^{1-p'}\dy\right)^{\!\frac{r}{q'}}\hspace{-5pt}v(t)^{1-p'}\dt\right)^{\!\frac{1}{r}}.
\]
Then $C \lesssim A_1+A_5+A_6$ hold true. However, the estimate $A_6 \lesssim C$ is valid if and only if $m\le q$.
\end{proposition}

The results for $p=1$ have a~different form, so, for the readers' convenience, they are singled out.

\begin{theorem}\label{T:main-p=1}
Let $q,m\in(0,\infty)$. Let $u,w,v$ be weights such that the pair $(u,w)$ is admissible with respect to $(m,q)$. Let
\[
C=\sup_{h\in\M_+}
\frac
{\left(\int_0^{\infty}\left(\int_t^\infty \left(\int_s^{\infty}h(y)\dy\right)^mu(s) \ds\right)^{\frac{q}{m}}w(t)\dt\right)^{\frac{1}{q}}}
{\int_0^{\infty}h(t)v(t)\dt}.
\]

\textup{(a)} Let $1\leq m<\infty$ and $1\leq q<\infty$. Then $C\approx \widetilde A_1$, where
\[
\widetilde A_1=
\esssup_{t\in(0,\infty)}\left(\int_0^tw(s)\left(\int_s^t u(y)\dy\right)^{\frac qm}\ds\right)^{\frac1q}
\frac 1{v(t)}.
\]

\textup{(b)} Let $1\leq m<\infty$ and $0<q<1$. Then $C\approx \widetilde A_2+\widetilde A_3$, where
\[
\widetilde A_2=
\left(\int_0^{\infty}
\left(\int_0^tw(s)\ds\right)^{-q'} w(t)
\esssup_{z\in(t,\infty)}
\left(\int_t^z u(y)\dy\right)^{-\frac{q'}m}
v(z)^{q'}\dt\right)^{-\frac 1{q'}}
\]
and
\[
\widetilde A_3=\hspace{-4pt}
\left(\int_0^{\infty}\hspace{-5pt}
\left(\int_0^t \!w(s)\!\left(\int_s^t \!u(y)\dy\!\right)^{\hspace{-2pt}\frac qm}\hspace{-7pt}\ds\!\right)^{\hspace{-5pt}-q'}
 \hspace{-8pt}w(t)
\esssup_{z\in(t,\infty)}\!
\left(\int_t^z \hspace{-4pt} u(y)\dy\!\right)^{\!\frac qm}
\hspace{-5pt}v(z)^{q'}\dt\!\right)^{\hspace{-6pt}-\frac 1{q'}}\hspace{-5pt}.
\]

\textup{(c)} Let $0<m<1$ and $1\leq q$. Then $C\approx \widetilde A_1+\widetilde A_4$, where
\[
\widetilde A_4=\!\!
\sup_{t\in(0,\infty)}\hspace{-5pt}\left(\int_0^t \! w(s)\ds\! \right)^{\! \frac1q}\!\!
\left(\int_t^{\infty}\hspace{-5pt}\left(\int_t^z \! u(y)\dy\right)^{\!\frac m{1-m}} \hspace{-6pt} u(z)\esssup_{y\in(z,\infty)} v(y)^{-\frac{m}{1-m}}\dz\right)^{\!\!\frac{1-m}{m}}\hspace{-6pt}.
\]

\textup{(d)} Let $0<m<1$, $1<p$ and $0<q<1$. Then $C\approx \widetilde A_3+\widetilde A_5$, where
\begin{align*}
\widetilde A_5 &=
\left( \vphantom{\Bigg(^\frac12} \int_0^{\infty}
\left(\int_0^tw(s)\ds\right)^{-q'}w(t) \right. \\
& \qquad\times\left.
\left(\int_t^{\infty}\left(\int_t^z u(y)\dy\right)^{\frac m{1-m}} \hspace{-5pt} u(z)\esssup_{y\in(z,\infty)}v(y)^{-\frac{m}{1-m}}\dz\right)^{\hspace{-4pt}-\frac{q'(1-m)}{m}}\hspace{-8pt}\dt
\right)^{\hspace{-4pt}-\frac 1{q'}}\hspace{-6pt}.
\end{align*}
\end{theorem}

Finally, as a~by-product in the proof of Theorem \ref{T:main}, one obtains a~characterization of the quantity~$C$ in \eqref{E:C} by discrete conditions. Typically, conditions characterizing weighted inequalities may be expressed in different but mutually equivalent ways. Finding the equivalent conditions and further dealing with them is in general considerably easier if a~discrete representation is at one's disposal. For this reason, we list the discrete conditions in the following corollary. (We omit the case $p=1$.)\\

However, before stating the result in the discrete form, we need to summarize at least the main elements of the discretization method which are going to be needed.  Let $(u,v)$ be an~admissible pair of weights with respect to $(m,q)$. Define $K=0$ if $\vp(\infty)<\infty$, and $K=\infty$ if $\vp(\infty)=\infty$. We denote $\K=\{k\in\Z\cup\{\infty\},\ k\le K\}$, and for any $k\in\K\setminus\{K\}$ we define $\Delta_k=[t_k,t_{k+1}]$. Then there exist sets $\mathbb K_1$ and $\mathbb K_2$ such that $\K_1\cap \K_2=\varnothing$, $\K=\K_1\cup \K_2$, and a sequence $\{t_k\}_{k\in\K}$ having the following properties:
	\[
		0<t_{k-1}\leq t_k<\infty \text{ for every } k\in\K\setminus\{K\}; \qquad t_K=\infty;
	\]
\begin{equation}\label{E:copson-5}
	\frac{2^{\frac{q}{m}+1}}{2^{\frac{q}{m}+1}\!-\!1} \int_{\Delta_{k-1}} \hspace{-8pt} w(t)\dt \ge \! \int_0^{t_k} \! w(t)\dt\geq 2^{\frac{q}{m}+1} \! \int_0^{t_{k-1}} \!\!\! w(t)\dt \quad \text{for all } k\in\K;
\end{equation}
\begin{equation}\label{E:copson-6}
\vp(t_k)^q\geq 2^{\frac{q}{m}+1}\vp(t_{k-1})^q \quad \text{for all } k\in\K;
\end{equation}
\begin{equation}\label{E:copson-7}
\int_0^{t_k}w(t)\dt=2^{\frac{q}{m}+1}\int_0^{t_{k-1}}w(t)\dt \quad \text{for all } k\in\K_1;
\end{equation}
\begin{equation}\label{E:copson-8}
\vp(t_k)^q=2^{\frac{q}{m}+1}\vp(t_{k-1})^q \quad \text{for all } k\in\K_2;
\end{equation}
and
\begin{align}
	\vp^q(t) & \lesssim \int_{\Delta_{k-3}} \hspace{-8pt} w(s)\ds \lt \int_{\Delta_{k-2}} \hspace{-8pt} u(y)\dy \rt^\frac qm \hspace{-8pt} + \int_{\Delta_{k-2}} \hspace{-8pt} w(s) \lt \int_s^{t_{k-1}} \hspace{-8pt} u(y) \dy \rt^\frac qm \ds \label{E:copson-9}\\
	& \qquad+ \int_{\Delta_{k-2}} \hspace{-8pt} w(s)\ds \lt \int_{t_{k-1}}^t \hspace{-8pt} u(y)\dy \rt^\frac qm \ \text{for all } k\in\K \text{ and } t\in \Delta_{k-1}.
\nonumber
\end{align}
The sequence $\{t_k\}_{k\in\K}$ is called the \textit{discretizing sequence}. We note that the admissibility of the pair $(u,w)$ guarantees that $\lim_{k\to-\infty}t_k=0$, and $\lim_{k\to \infty}t_k=\infty$ if $K=\infty$. As a consequence, $\int_0^{\infty} f(t)\dt=\sum_{k\in\K}\int_{\Delta_{k-1}}f(t)\dt$ holds whenever $f$ is a~non-negative measurable function on $(0,\infty)$.

A general reference for all that has been said in the last paragraph is~\cite[Theorems 3.4, 3.7]{copson}, see also~\cite{KGGE,CPSS,GP2} and more. Now we can state the promised corollary.

\begin{corollary}\label{Cor}
	Let $p\in(1,\infty)$ and $q,m\in(0,\infty)$. Let $u,w,v$ be weights such that the pair $(u,w)$ is admissible with respect to $(m,q)$, and let $\{t_k\}_{k\in\K}$ be the corresponding discretizing sequence. Let $C$ be given by \eqref{E:C}.
	
	\textup{(a)} Let $1<p\leq m<\infty$ and $p\leq q<\infty$. Then $C\approx D_1$, where
		\[
		D_1=\sup_{k\in\K}\sup_{t\in\Delta_{k-1}}
		\vp(t)\left(\int_t^{t_k}v(s)^{1-p'}\ds\right)^{\frac 1{p'}}.
		\]
		
	\textup{(b)} Let $1<p\leq m<\infty$ and $0<q<p$. Define $r= \frac{pq}{p-q}$. Then $C\approx D_2$, where
		\[
		D_2 = \left( \sum_{k\in\K} \sup_{t\in\Delta_{k-1}} \vp(t)^r \left(\int_{t}^{t_k}v(s)^{1-p'}\ds\right)^{\frac{r}{p'}}\right)^{\frac{1}{r}}.
		\]
			
	\textup{(c)} Let $0<m<p<\infty$ and $1<p\leq q$. Then $C\approx D_1+D_3$, where
		\[
			D_3
			=
			\sup_{k\in\K}
			\left(\int_{\Delta_{k-1}}
			\vp(t)^{\frac{mp}{p-m}}
			\left(\int_{t}^{t_k}v(s)^{1-p'}\ds\right)^{\frac{p(m-1)}{p-m}}v(t)^{1-p'}\dt\right)^{\frac{p-m}{{mp}}}.
		\]
	If, moreover, $m\geq1$ or $\int_{t}^{\infty}v(s)^{1-p'}\ds<\infty$ for every $t\in(0,\infty)$, then $D_1 \lesssim D_3$ and thus $C\approx D_3$.\\
	
	\textup{(d)} Let $0<m<p<\infty$, $1<p$ and $0<q<p<\infty$. Define $r= \frac{pq}{p-q}$. Then $C\approx D_1 + D_4$, where
		 \[
			 D_4=
			 \left(
			 \sum_{k\in\K}\left(\int_{\Delta_{k-1}}
			 \vp(t)^{\frac{mp}{p-m}}
			 \left(\int_{t}^{t_k}v(s)^{1-p'}\ds\right)^{\frac{p(m-1)}{p-m}}
			 v(t)^{1-p'}\dt\right)^{\frac{q(p-m)}{m(p-q)}}
			 \right)^{\frac{1}{r}}.
		 \]
	If, moreover, $m\ge 1$ or $\int_{t}^{\infty}v(s)^{1-p'}\ds<\infty$ for every $t\in(0,\infty)$, then $D_1 \lesssim D_4$ and thus $C\approx D_4$.
\end{corollary}

\section{Proofs}

We begin with recalling three known lemmas that will be needed in the proof of the main theorem, for reader's convenience stated in full. For the proofs, see e.g.~\cite{GHS,GP,K-complut}, etc.

\begin{lemma}[discrete H\"older's inequality]\label{L:40}
  Let $k_{\min},k_{\max}\in\Z\cup\{\pm\infty\}$ be such that $k_{\min}<k_{\max}$. Let $\{a_k\}_{k=k_{\min}}^{k_{\max}}$ and $\{b_k\}_{k=k_{\min}}^{k_{\max}}$ be two non-negative sequences. Assume that $0<q<p<\infty$. Then
    \[
      \left(\sum_{k=k_{\min}}^{k_{\max}} a_k^q b_k\right)^\frac1q \le \left(\sum_{k=k_{\min}}^{k_{\max}} a_k^p\right)^\frac1p \left( \sum_{k=k_{\min}}^{k_{\max}} b_k^\frac{p}{p-q}\right)^\frac{p-q}{pq}.
    \]
  Moreover, there exists a~non-negative sequence $\{c_k\}_{k=k_{\min}}^{k_{\max}}$ such that \linebreak $\sum_{k=k_{\min}}^{k_{\max}} c_k^p=1$ and
    \[
      \left( \sum_{k=k_{\min}}^{k_{\max}} b_k^\frac{p}{p-q}\right)^\frac{p-q}{pq} = \left(\sum_{k=k_{\min}}^{k_{\max}} c_k^q b_k\right)^\frac1q.
    \]
\end{lemma}

\begin{lemma}\label{L:20}
  Let $0<\alpha<\infty$ and $1<D<\infty$. Then there exists a~constant $C_{\alpha,D}\in(0,\infty)$ such that for any $k_{\min},k_{\max}\in\Z\cup\{\pm\infty\}$, $k_{\min}<k_{\max}$, and any two non-negative sequences $\{b_k\}_{k=k_{\min}}^{k_{\max}}$ and $\{c_k\}_{k=k_{\min}}^{k_{\max}}$, satisfying $b_{k+1}\ge D\, b_k$  for all $k\in\Z,\ k_{\min}\le k <k_{\max}$, there holds
    \[
      \sum_{k=k_{\min}}^{k_{\max}} \left( \sum_{m=k}^{k_{\max}} c_m \right)^\alpha b_k  \le C_{\alpha,D} \sum_{k=k_{\min}}^{k_{\max}} c_k^\alpha b_k,
    \]
    \[
      \sum_{k=k_{\min}}^{k_{\max}} \left(\sup_{k\le m \le k_{\max}} \!\! c_m \right)^\alpha b_k  \le C_{\alpha,D}  \sum_{k=k_{\min}}^{k_{\max}} c_k^\alpha b_k
    \]
  and
    \[
      \sup_{k_{\min}\le k \le k_{\max}} \left( \sum_{m=k}^{k_{\max}} c_m \right)^\alpha b_k  \le C_{\alpha,D} \sup_{k_{\min}\le k \le k_{\max}} c_k^\alpha b_k.
    \]
\end{lemma}

\begin{lemma}[Hardy's inequality]\label{L:35}
   Let $a,b\in[0,\infty]$ and let $\eta,\varrho$ be weights.

   {\rm(i)}
   Let $1<\alpha\le \beta<\infty$. Then the inequality
   \begin{multline*}
   \lt \int_a^b \lt \int_t^b h(s)\ds \rt^\beta \varrho(t)\dt \rt^\frac1\beta \\ \lesssim \lt \int_a^b h^\alpha(t)\eta(t) \dt \rt^\frac1\alpha \sup_{t\in(a,b)} \lt \int_a^t \varrho(s)\ds \rt^\frac1\beta \lt \int_t^b \eta^{1-\alpha'}(s) \ds \rt^\frac1{\alpha'}
   \end{multline*}
   holds for all $h\in\MM(a,b)$. Moreover, there exists a~function $g\in\MM(a,b)$ such that $\int_a^b g^\alpha \eta = 1$ and
   \[
   \sup_{t\in(a,b)} \lt \int_a^t \varrho(s)\ds \rt^\frac1\beta \lt \int_t^b \eta^{1-\alpha'}(s) \ds \rt^\frac1{\alpha'} \lesssim  \lt \int_a^b \lt \int_t^b g(s)\ds \rt^\beta \varrho(t)\dt \rt^\frac1\beta.
   \]

   {\rm(ii)}
   Let $1<\alpha<\infty$, $0<\beta<\alpha$ and $\int_t^b \eta^{1-\alpha'}(s)\ds <\infty$ for all $t\in(a,b)$. Then the inequality
   \begin{multline*}
	   \lt \int_a^b \lt \int_t^b h(s)\ds \rt^\beta \varrho(t)\dt \rt^\frac1\beta \\
	   \lesssim \lt \int_a^b \! h^\alpha(t)\eta(t) \dt \! \rt^{\!\!\frac1\alpha} \hspace{-4pt} \lt \!\int_a^b \!\! \lt \int_a^t \!\varrho(s) \ds \!\rt^{\!\!\frac{\alpha}{\alpha-\beta}} \hspace{-6pt} \lt \int_t^b \eta^{1-\alpha'}(s) \ds \rt^{\hspace{-6pt}\frac{\alpha(\beta-1)}{\alpha-\beta}} \hspace{-8pt} \eta^{1-\alpha'}(t) \dt \!\rt^{\!\!\!\frac{\alpha-\beta}{\alpha\beta}}
   \end{multline*}
   holds for all $h\in\MM(a,b)$. Moreover, there exists a~function $g\in\MM(a,b)$ such that $\int_a^b g^\alpha \eta = 1$ and
   \begin{multline*}
   \lt \int_a^b \lt \int_a^t \varrho(s) \ds \rt^\frac{\alpha}{\alpha-\beta}  \lt \int_t^b \eta^{1-\alpha'}(s) \ds \rt^\frac{\beta}{\alpha-\beta} \eta^{1-\alpha'}(t) \dt \rt^\frac{\alpha-\beta}{\alpha\beta} \\ \lesssim  \lt \int_a^b \lt \int_t^b g(s)\ds \rt^\beta \varrho(t)\dt \rt^\frac1\beta.
   \end{multline*}
\end{lemma}

The rest of the paper contains the detailed proof of the main results.

We are using the traditional terms ``sufficiency'' and ``necessity'', meaning that a~certain condition, e.g.~$A_1<\infty$, is sufficient and necessary, respectively, for $C$ being finite.

\begin{proof}[Proof of Theorem~\ref{T:main}] We begin with discretization of the quantity on the left hand side of~\eqref{E:main}. By~\cite[Theorem~3.8]{copson} applied to appropriate parameters, for any $h\in\MM$ one has
\Bdef{1}\Bdef{2}
  \begin{equation}\label{E:diskr}
    \left(\int_0^{\infty}\left(\int_t^\infty \left(\int_s^{\infty}h(y)\dy\right)^mu(s) \ds\right)^{\frac{q}{m}}w(t)\dt\right)^{\frac{1}{q}}
    \approx
    \B{1}+\B{2},
  \end{equation}
where
\[
\B{1}=\left(\sum_{k\in\K}\vp(t_{k-1})^q\left(\int_{\Delta_{k-1}}h(t)\dt\right)^q\right)^{\frac 1q}
\]
and
\[
\B{2}=\left(\sum_{k\in\K}\left(\int_{\Delta_{k-1}}\vp(t)^{m-1}\vp'(t)\left(\int_t^{t_{k}}h(s)\ds\right)^m\dt\right)^{\frac{q}{m}}\right)^{\frac1q}.
\]

We shall now prove the sufficiency part of the theorem, that is, upper bounds for $C$ in terms of the $A$'s. We shall distinguish several cases. At first, we make the following two observations which are valid in both cases (a) and (b).

Assume that $1<p\leq m<\infty$. Then, using the H\"older inequality, for each $h\in\MM$ we have
\begin{equation}\label{E:Hoelder-obyc}
\left(\int_{\Delta_{k-1}}h(t)\dt\right)^q
\le
\left(\int_{\Delta_{k-1}}h(t)^pv(t)\dt\right)^{\frac qp}
\left(\int_{\Delta_{k-1}}v(t)^{1-p'}\dt\right)^{\frac q{p'}}.
\end{equation}
 In addition to that, the Hardy inequality (Lemma \ref{L:35}(i)) yields
\begin{align}
&\left(\int_{\Delta_{k-1}} \hspace{-10pt} \vp(t)^{m-1}\vp'(t)\left(\int_t^{t_{k}} \hspace{-6pt} h(s)\ds\right)^m \hspace{-6pt} \dt\right)^{\frac{1}{m}}\label{E:Hardy-ab}\\
&\quad\lesssim
\sup_{t\in\Delta_{k-1}}\!\!
\left(\int_{t_{k-1}}^t \hspace{-6pt} \vp(s)^{m-1}\vp'(s)\ds\right)^{\!\frac1m}
\hspace{-6pt}\left(\int_t^{t_{k}} \hspace{-6pt} v(s)^{1-p'}\ds\right)^{\!\frac1{p'}}
\hspace{-4pt}\left(\int_{\Delta_{k-1}} \hspace{-10pt} h(s)^p v(s)\ds\right)^{\!\frac{1}{p}}.\nonumber
\end{align}

(a)  Let $1<p\leq m<\infty$ and $p\leq q<\infty$ and suppose that $A_1<\infty$.
Combining \eqref{E:Hoelder-obyc} with \eqref{E:Hardy-ab}, we get
\begin{align*}
&\B{1}+\B{2}\\
&\lesssim
\left(\sum_{k\in\K}\vp(t_{k-1})^q
\left(\int_{\Delta_{k-1}}h(t)^pv(t)\dt\right)^{\frac qp}
\left(\int_{\Delta_{k-1}}v(t)^{1-p'}\dt\right)^{\frac q{p'}}\right.\\
&\quad \left. +
\sum_{k\in\K}\left(\int_{\Delta_{k-1}} \hspace{-12pt} h(t)^pv(t)\dt\!\right)^{\!\frac qp}
\hspace{-6pt} \sup_{t\in\Delta_{k-1}} \!\! \left(\int_{t_{k-1}}^t\ \hspace{-12pt} \vp(s)^{m-1}\vp'(s)\ds\!\right)^{\hspace{-4pt}\frac qm}
\hspace{-6pt} \left(\int_t^{t_k} \!\!v(s)^{1-p'}\ds\!\right)^{\!\frac q{p'}}\!\right)^{\hspace{-6pt}\frac 1q}\\
&\lesssim
\sup_{k\in\K} \left(\sum_{k\in\K}\left(\int_{\Delta_{k-1}}h(t)^pv(t)\dt\right)^{\!\!\frac qp}\right)^{\!\!\frac 1q} \\
&\quad\times\!\!
\left(\!\vp(t_{k-1})\!\left(\int_{\Delta_{k-1}} \hspace{-12pt} v(t)^{1-p'}\!\dt\!\right)^{\!\!\frac 1{p'}} 
\hspace{-8pt} + \hspace{-4pt}
\sup_{t\in\Delta_{k-1}}\!\!
\left(\int_{t_{k-1}}^t \hspace{-11pt} \vp(s)^{m-1}\vp'(s)\ds\!\right)^{\!\!\frac 1m}
\hspace{-6pt}\left(\int_t^{t_k} \hspace{-6pt} v(s)^{1-p'}\!\ds\!\right)^{\!\!\frac 1{p'}}\!\right)\!\!.
\end{align*}
Therefore,
\begin{align*} 
\B{1}+\B{2}
&\ls
\sup_{k\in\K}\sup_{t\in\Delta_{k-1}}
\left(\vp(t_{k-1})+\left(\int_{t_{k-1}}^t\vp(s)^{m-1}\vp'(s)\ds\right)^{\frac 1m}\right)\\
&\hspace{70pt} \times
\left(\int_t^{t_k}v(s)^{1-p'}\ds\right)^{\frac 1{p'}}
\left(\sum_{k\in\K}\int_{\Delta_{k-1}}h(t)^pv(t)\dt\right)^{\frac 1p}\\
&\approx
\sup_{k\in\K}\sup_{t\in\Delta_{k-1}}
\vp(t)\left(\int_t^{t_k}v(s)^{1-p'}\ds\right)^{\frac 1{p'}}
\left(\int_0^{\infty}h(t)^pv(t)\dt\right)^{\frac{1}{p}}\\
&\lesssim
A_1 \left(\int_0^{\infty}h(t)^pv(t)\dt\right)^{\frac{1}{p}}.
\end{align*}
We have proved that
	\[
		C\lesssim A_1.
	\]

(b) Let $1<p\leq m<\infty$ and $0<q<p$. Suppose that $A_2+A_3<\infty$.  A~subsequent use of \eqref{E:Hoelder-obyc} and \eqref{E:Hardy-ab}, the discrete H\"older inequality (Lemma \ref{L:40}) and integration by parts yields
\begin{align*}
\B{1}+\B{2}
&\lesssim
\left(\sum_{k\in\K}\sup_{t\in\Delta_{k-1}}\left(\varphi(t_{k-1})+\left(\int_{t_{k-1}}^{t}
\vp(s)^{m-1}\vp'(s)\ds\right)^{\frac 1m}\right)^q \right. \\
& \qquad \times \left. \left(\int_{t}^{t_k}v(s)^{1-p'}\ds\right)^{\frac{q}{p'}}
\lt\int_{\Delta_{k-1}}h(t)^{p}v(t)\dt\right)^{\frac{q}{p}} \right)^{\frac{1}{q}}\\
&\lesssim
\left(\sum_{k\in\K}\sup_{t\in\Delta_{k-1}}\left(\varphi(t_{k-1})+\left(\int_{t_{k-1}}^{t}
\vp(s)^{m-1}\vp'(s)\ds\right)^{\frac 1m}\right)^r \right.\\ 
& \qquad \times \left.
\left(\int_{t}^{t_k}v(s)^{1-p'}\ds\right)^{\frac{r}{p'}}\right)
^{\frac{1}{r}} 
\left(\sum_{k\in\K}
\int_{\Delta_{k-1}}h(t)^{p}v(t)\dt\right)^{\frac{1}{p}}\\
&
\lesssim
\left(\sum_{k\in\K}\sup_{t\in\Delta_{k-1}}
\vp(t)^r\left(\int_{t}^{t_k}v(s)^{1-p'}\ds\right)^{\frac{r}{p'}}\right)^{\frac{1}{r}}
\left(\int_0^{\infty}h(t)^pv(t)\dt\right)^{\frac{1}{p}}\\
&
=D_2\left(\int_0^{\infty}h(t)^pv(t)\dt\right)^{\frac{1}{p}}.
\end{align*}
(In here, we use the notation from Corollary \ref{Cor}.) By~\eqref{E:copson-9}, one has
\Bdef{21}\Bdef{22}\Bdef{23}
\begin{align*}
	D_2
		&\approx
			\left( \sum_{k\in\K}
			\left( \int_{\Delta_{k-3}} w(s)\ds \right)^{\frac{r}{q}}
			\left( \int_{\Delta_{k-2}} u(y)\dy \right)^{\frac{r}{m}}
			\left(\int_{\Delta_{k-1}} v(s)^{1-p'}\ds\right)^{\frac{r}{p'}} \right.\\
		&\qquad +
			\sum_{k\in\K}
			\left(\int_{\Delta_{k-2}}w(s)
			\left(\int_{s}^{t_{k-1}}u(y)\dy\right)^{\frac{q}{m}}
			\ds\right)^{\frac{r}{q}}
			\left( \int_{\Delta_{k-1}} v(s)^{1-p'}\ds\right)^{\frac{r}{p'}}\\
		&\qquad +
			\left.\sum_{k\in\K}
			\left(\int_{\Delta_{k-2}} \hspace{-10pt} w(s)\ds \right)^{\frac{r}{q}}
			 \sup_{t\in\Delta_{k-1}} \left( \int_{t_{k-1}}^{t} \hspace{-10pt} u(y)\dy \right)^{\frac{r}{m}}
			\hspace{-4pt}\left(\int_t^{t_k} \hspace{-4pt} v(s)^{1-p'}\ds\right)^{\frac{r}{p'}}
			\right)^{\frac{1}{r}}\\
		& = \B{21} + \B{22} + \B{23}.
\end{align*}
	It is checked easily that  $\B{21}+\B{23} \lesssim A_2$. Recalling that $\frac{rq}{pm} + \frac qm = \frac rm $, for the term $\B{22}$ we get
		\begin{align*}
			\B{22}
				&   \approx
					\lt
					\sum_{k\in\K}
					\int_{\Delta_{k-2}} \hspace{-4pt}
					\left( \int_{t_{k-2}}^s \hspace{-10pt} w(z)
					\left( \int_{z}^{t_{k-1}}  \hspace{-10pt} u(y)\dy \right)^{\frac{q}{m}}
					\hspace{-6pt}\dz \right)^{\frac rp}
					w(s) \left(\int_{s}^{t_{k-1}} \hspace{-10pt} u(y)\dy\right)^{\frac{q}{m}}  \ds \right.\\
				&   \left.\quad\times
					\left( \int_{\Delta_{k-1}} v(s)^{1-p'}\ds\right)^{\frac{r}{p'}}
					\rt^\frac1r\\
				&	\approx
					\lt
					\sum_{k\in\K}
					\int_{\Delta_{k-2}}\hspace{-4pt}
					\left( \int_{t_{k-2}}^s \hspace{-10pt} w(z)
					\left( \int_{z}^s u(y)\dy \right)^{\frac{q}{m}}
					\hspace{-6pt}\dz \right)^{\frac rp}
					w(s) \left(\int_{s}^{t_{k-1}} \hspace{-10pt} u(y)\dy\right)^{\frac{q}{m}}  \ds \right.\\
				&   \left.\quad\times
					\left( \int_{\Delta_{k-1}} v(s)^{1-p'}\ds\right)^{\frac{r}{p'}}
					\rt^\frac1r\\
				&   \quad +
					\lt
					\sum_{k\in\K}
					\int_{\Delta_{k-2}} \hspace{-6pt}
					\left( \int_{t_{k-2}}^s\hspace{-10pt} w(z)\!
					\dz \!\right)^{\hspace{-2pt}\frac rp}
					\!w(s)\! \left(\int_{s}^{t_{k-1}} \hspace{-14pt} u(y)\dy\!\right)^{\!\frac rm} \hspace{-8pt} \ds
					\left( \int_{\Delta_{k-1}} \hspace{-12pt} v(s)^{1-p'}\ds\!\right)^{\hspace{-4pt}\frac{r}{p'}}
					\rt^{\!\!\frac1r} \\
				&   \lesssim A_3 + A_2.
		\end{align*}
Hence, we have proved
	\[
		D_2 \lesssim A_2 + A_3.
	\]
Therefore, we also obtain
\[
\B{1}+\B{2}
\lesssim
(A_2+A_3)
\left(\int_0^{\infty}h(t)^p v(t)\dt\right)^{\frac{1}{p}}
\]
and thus, finally,
	\[ C \lesssim A_2 +A_3. \]
This completes the sufficiency part in the case (b).\\

Before we turn our attention to the remaining two cases, we make a~universal observation valid for all parameters $m,q\in(0,\infty)$ and $p\in(1,\infty)$: Since, throughout the paper, we assume that $\vp(t)>0$ for every $t\in(0,\infty)$, the following implication is true:
   \begin{equation}\label{E:psourtek}
     A_1<\infty \qquad \Longrightarrow \qquad \int_{t}^{\infty}v(s)^{1-p'}\ds<\infty \text{\ for every\ } t\in(0,\infty).
   \end{equation}
This simple fact will be used on multiple occasions to ensure correctness of certain calculations.

Now we present an estimate which is relevant for cases (c) and (d). Assume that $0<m<p<\infty$, $p>1$ and $A_1<\infty$. Then the Hardy inequality (Lemma \ref{L:35}(ii)) with the observation \eqref{E:psourtek} implies that
\begin{align*}
&\left(\int_{\Delta_{k-1}}\vp(t)^{m-1}\vp'(t)
\left(\int_{t}^{t_k}h(s)\ds\right)^{m}\dt\right)^{\frac{1}{m}} \\
&\quad \lesssim
\left(\int_{\Delta_{k-1}} \!\!
\left(\int_{t_{k-1}}^{t} \! \vp(s)^{m-1}\vp'(s)\ds\right)^{\!\!\frac{p}{p-m}} \!\!\!
\left(\int_{t}^{t_k} \! v(s)^{1-p'}\ds\right)^{\!\!\frac{p(m-1)}{p-m}} \hspace{-8pt} v(t)^{1-p'}\dt\right)^{\hspace{-6pt}\frac{p-m}{{mp}}}\\
&\qquad\qquad \times \left(\int_{\Delta_{k-1}}h(s)^p v(s)\ds\right)^{\frac{1}{p}}\!. \nonumber
\end{align*}
	This estimate together with \eqref{E:Hoelder-obyc} yields
\begin{align*}
&\B{1}+\B{2} \\
&\lesssim
\left(\sum_{k\in\K}\vp(t_{k-1})^q
\left(\int_{\Delta_{k-1}}v(s)^{1-p'}\ds\right)^{\frac{q}{p'}}
\left(\int_{\Delta_{k-1}}h(s)^pv(s)\ds\right)^{\frac{q}{p}}
\right)^{\frac{1}{q}}\\
&\quad+\!
\left(\sum_{k\in\K} \!
\left(\int_{\Delta_{k-1}} \hspace{-5pt}
\left(\int_{t_{k-1}}^{t} \hspace{-8pt} \vp(s)^{m-1} \vp'(s)\ds\right)^{\hspace{-2pt}\frac{p}{p-m}} \hspace{-6pt}
\left(\int_{t}^{t_k} \hspace{-6pt} v(s)^{1-p'}\ds\right)^{\hspace{-6pt}\frac{p(m-1)}{p-m}} \hspace{-14pt} v(t)^{1-p'}\dt\right)^{\hspace{-6pt}\frac{(p-m)q}{{mp}}}\right.\\
&\quad\qquad\left.\times
\left(\int_{\Delta_{k-1}}h(s)^pv(s)\ds\right)^{\frac{q}{p}}
\right)^{\frac{1}{q}}.
\end{align*}
Integrating by parts, we get
\Bdef{3}
	\begin{equation}\label{E:126}
		\B{1}+\B{2}\lesssim \B{3},
	\end{equation}
where
	\[
		\B{3}= \hspace{-5pt}
			\left( \! \sum_{k\in\K}\hspace{-4pt}
			\left(\!\int_{\Delta_{k-1}}
			\hspace{-14pt} \vp(t)^\frac{pm}{p-m} \!\!
			\left(\int_{t}^{t_k} \hspace{-6pt} v(s)^{1-p'}\!\ds\!\right)^{\hspace{-6pt}\frac{p(m-1)}{p-m}} \hspace{-18pt} v(t)^{1-p'}\!\dt\!\right)^{\hspace{-6pt}\frac{(p-m)q}{{mp}}}
			\hspace{-10pt} \left(\int_{\Delta_{k-1}} \hspace{-14pt} h(s)^pv(s)\ds\!\right)^{\!\!\frac{q}{p}}
			\!\right)^{\!\!\frac{1}{q}}\!\!.
	\]

(c) Let $0<m<p<\infty$ and $1<p\leq q$. Assume that $A_1+A_4<\infty$. Recall that the implication \eqref{E:psourtek} assures that the estimate \eqref{E:126} is correct. Then the fact that $q\geq p$ guarantees that
\[
\B{3}\leq D_3
\left(\sum_{k\in\K}\int_{\Delta_{k-1}}h(s)^p v(s)\ds\right)^{\frac{1}{p}},
\]
where $D_3$ is defined in Corollary \ref{Cor}. Thus, by~\eqref{E:copson-9}, we get
\begin{align*}
D_3
	&\lesssim
		\sup_{k\in\K}
		\left(\int_{\Delta_{k-3}} w(t)\dt\right)^{\frac{1}{q}}
		\left(\int_{\Delta_{k-2}} u(t)\dt\right)^{\frac{1}{m}}
		\left(\int_{\Delta_{k-1}} v(t)^{1-p'}\dt\right)^{\frac{1}{p'}}\\
	& \quad +
		\sup_{k\in\K} \left(\int_{\Delta_{k-2}} w(t)\left(\int_t^{t_{k-1}}u(s)\ds\right)^{\frac{q}{m}}\dt\right)^{\frac{1}{q}}
		\left(\int_{\Delta_{k-1}} v(t)^{1-p'}\dt \right)^{\frac{1}{p'}}\\
	& \quad+
		\sup_{k\in\K} \left(\int_{\Delta_{k-2}} \hspace{-14pt} w(t)\dt \!\right)^{\hspace{-2pt}\frac{1}{q}}
		\hspace{-4pt} \left(\int_{\Delta_{k-1}} \hspace{-6pt} \left(\int_{t_{k-1}}^{t} \hspace{-12pt} u(s)\ds \!\right)^{\!\!\frac{p}{p-m}}
		\hspace{-6pt} \left(\int_{t}^{t_{k}} \hspace{-8pt} v(s)^{1-p'}\!\ds \!\right)^{\hspace{-4pt}\frac{p(m-1)}{p-m}}
		\hspace{-16pt} v(t)^{1-p'}\!\dt\!\right)^{\hspace{-6pt}\frac{p-m}{pm}}\\
	&\lesssim
		A_1+A_4,
\end{align*}
	We have shown that
		\[
			\B{3} \lesssim (A_1+ A_4) \left( \int_0^\infty h(s)^p v(s)\ds\right)^{\frac{1}{p}}.
		\]
From \eqref{E:diskr} and \eqref{E:126} we get
	\[
		C \lesssim A_1 + A_4,
	\]
proving the sufficiency part of the assertion (c).

(d) Let $0<m<p<\infty$, $p>1$ and $0<q<p<\infty$. Assume that $A_3+A_5<\infty$. An~observation similar to \eqref{E:psourtek} yields that if $A_3<\infty$, then $\int_t^\infty v(s)^{1-p'}\ds <\infty$ for all $t>0$. Hence, it is justified to use the estimate \eqref{E:126}.
	
	From \eqref{E:126} and the discrete H\"older inequality (Lemma \ref{L:40}), it follows that
\[
\B{3}\leq D_4 \left(\sum_{k\in\K}\int_{\Delta_{k-1}}h(t)^pv(t)\dt\right)^{\frac{1}{p}},
\]
where $D_4$ is the expression defined in Corollary \ref{Cor}. By~\eqref{E:copson-9}, we get
\begin{align*}
	D_4
		&\lesssim
			\left(\sum_{k\in\K}
			\left(\int_{\Delta_{k-3}} w(t)\dt\right)^{\frac{r}{q}}
			\left(\int_{\Delta_{k-2}} u(s)\ds\right)^{\frac{r}{m}}
			\left(\int_{\Delta_{k-1}} v(z)^{1-p'}\dz\right)^{\frac{r}{p'}}
			\right)^\frac1r\\
		&\quad
			+ \left(\sum_{k\in\K}
			\left( \int_{\Delta_{k-2}} \! w(t) \left( \int_{t}^{t_{k-1}} \! u(s)\ds \right)^{\frac{q}{m}} \dt \right)^{\frac{r}{q}}
			\left( \int_{\Delta_{k-1}} \! v(z)^{1-p'}\dz \right)^{\frac{r}{p'}}
			\right)^\frac1r\\
		&\quad
			+ \left( \sum_{k\in\K}
			\left(\int_{\Delta_{k-2}} w(t)\dt \right)^{\frac{r}{q}} \right.\\
		&\quad\quad\left.\times
			\left(\int_{\Delta_{k-1}}\!\!
			\left(\int_{t_{k-1}}^{t} \hspace{-6pt} u(s)\ds \right)^{\!\!\frac{p}{p-m}}
			\hspace{-6pt} \left(\int_{t}^{t_{k}} \hspace{-6pt} v(z)^{1-p'}\dz \right)^{\!\!\frac{p(m-1)}{p-m}}
			\hspace{-6pt} v(t)^{1-p'}\dt\right)^{\!\!\frac{q(p-m)}{m(p-q)}}
			\right)^{\!\!\frac{1}{r}}\!\!,
\end{align*}
and, consequently,
\begin{align*}
	D_4		&\lesssim
			A_5
			+ \!\left(\sum_{k\in\K}
			 \int_{\Delta_{k-2}} \hspace{-6pt} \lt  \int_{t_{k-2}}^y \hspace{-8pt} w(t) \left( \int_{t}^{t_{k-1}}  \hspace{-12pt} u(s)\ds \right)^{\!\!\frac qm} \hspace{-6pt} \dt \right)^{\!\!\frac rp}  w(y) \left( \int_{y}^{t_{k-1}} \hspace{-12pt} u(s)\ds \right)^{\!\!\frac qm} \hspace{-6pt}\dy \right.\\
		& \qquad\qquad\times\left.
			\left( \int_{\Delta_{k-1}} v(t)^{1-p'}\dt \right)^{\frac{r}{p'}}
			\right)^\frac1r\\
		&\lesssim
			A_5
			+ \!\left(\sum_{k\in\K}
			\int_{\Delta_{k-2}} \hspace{-6pt} \lt \int_{t_{k-2}}^y \hspace{-8pt} w(t) \left( \int_{t}^{y} \hspace{-4pt} u(s)\ds \right)^{\!\!\frac qm}  \hspace{-6pt} \dt \right)^{\!\!\frac rp} w(y) \left( \int_{y}^{t_{k-1}} \hspace{-12pt} u(s)\ds \right)^\frac qm \hspace{-6pt} \dy \right.\\
		&\qquad\qquad\left.\times
			\left( \int_{\Delta_{k-1}} v(t)^{1-p'}\dt \right)^{\frac{r}{p'}}
			\right)^\frac1r\\
		& \qquad +\!
			\left(\sum_{k\in\K}
			\int_{\Delta_{k\!-\!2}} \hspace{-6pt} \lt \int_{t_{k-2}}^y \hspace{-10pt} w(t) \dt \! \right)^{\hspace{-3pt}\frac rp} \! w(y)\!
			\left( \int_{y}^{t_{k-1}} \hspace{-12pt} u(s)\ds \! \right)^{\hspace{-4pt}\frac rm} \hspace{-6pt} \dy
			\left( \int_{\Delta_{k-1}} \hspace{-12pt} v(t)^{1-p'}\!\dt \!\right)^{\!\!\frac{r}{p'}}
			\!\right)^{\!\!\frac1r}\\
		& \lesssim A_3 + A_5.
\end{align*}
This estimate shows that
	\[
		\B{3} \lesssim (A_3+A_5) \left( \int_0^\infty h(t)^p v(t)\dt \right)^{\frac{1}{p}},
	\]
and thus (by \eqref{E:diskr} and \eqref{E:126}) also
	\[
		C \lesssim A_3 + A_5.
	\]
Therefore, we have finished proving the sufficiency part.\\

We shall now turn our attention to the necessity part of the theorem, that is, of the lower bounds for~$C$. In general, we will make use of the saturation of the H\"older and Hardy inequalities which have been used in the proof of the sufficiency part.

The H\"older inequality is saturated in the following sense. If $k\in\K$ and $1<p<\infty$, there exists a~function $g_k \in \MM$ supported in $\Delta_{k-1}$, satisfying $\int_{\Delta_{k-1}}g_k(t)^pv(t)\dt=1$ and such that
  \begin{equation}\label{E:nutnost-2}
    \left(\int_{\Delta_{k-1}}v(s)^{1-p'}\ds\right)^{\frac1{p'}}
    =\int_{\Delta_{k-1}}g_k(t)\dt.
  \end{equation}

From now on, we assume that $C<\infty$. As a~first step, we will show that the estimate $A_1\le C$ holds, in fact, for any parameters $m,q\in(0,\infty)$ and $p\in(1,\infty)$, thus in all cases (a)--(d). To prove this claim, fix $t>0$ and find, similarly as above, a~function $\psi\in\MM$ supported in $(t,\infty)$, satisfying $\int_0^\infty \psi(t)^p v(t)\dt=1$ and such that
  \[
    \lt \int_t^\infty v(s)^{1-p'} \ds \rt^\frac1{p'} = \int_t^\infty \psi(t)\dt.
  \]
Now we may write
  \begin{align*}
      \vp(t)\left(\int_t^{\infty}v(s)^{1-p'}\ds\right)^{\frac1{p'}} \! & = \lt \int_0^t w(s) \lt \int_s^t u(y)\dy \rt^\frac qm \ds \rt^\frac1q \int_t^\infty \psi(t)\dt \\
            & \le \lt \int_0^\infty \!\!w(s) \lt \int_s^\infty \!\! u(y) \lt \int_y^\infty \!\!\psi(t) \dt \rt^m \!\!\dy \rt^\frac qm \!\ds \rt^\frac1q \\
            & \le C.
  \end{align*}
In the last step we used the fact that $\|\psi\|_{L^p(v)}=1$. Taking the supremum over $t>0$, we get
  \[
    A_1 \le C.
  \]
In particular, we have proved the necessity part of (a). Notice that the simple estimate used above can be applied only to the condition $A_1$. Dealing with the rest of the $A$'s is more difficult and requires the use of discretization. By discretization one can also prove that $$A_1\lesssim D_1 \lesssim C$$ for all $m,q\in(0,\infty)$ and $p\in(1,\infty)$. However, we omit this particular detail.\\

(b) Let $1<p\leq m<\infty$ and $0<q<p$. Recall the saturation of the Hardy inequality (Lemma \ref{L:35}(i)). It grants that for every $k\in\K$ there exists a~function $h_k\in\M_+$ supported in $\Delta_{k-1}$, satisfying $\int_{\Delta_{k-1}}h_k(t)^pv(t)\dt=1$ and such that
  \begin{multline}\label{E:nutnost-1}
    \sup_{t\in\Delta_{k-1}}
    \left(\int_{t_{k-1}}^t\vp(s)^{m-1}\vp'(s)\ds\right)^{\frac1m}
    \left(\int_t^{t_{k}}v(s)^{1-p'}\ds\right)^{\frac1{p'}}\\
    \lesssim
    \left(\int_{\Delta_{k-1}}\vp(t)^{m-1}\vp'(t)\left(\int_t^{t_{k}}h_k(s)\ds\right)^{m}\dt\right)^{\frac1m}.
  \end{multline}
We may write
\[
A_2\!
=\!\!
\left(\sum_{k\in\K}\int_{\Delta_{k-1}}
\hspace{-6pt} \left(\int_{0}^{t}w(z)\dz\!\right)^{\!\frac{r}{p}}
\hspace{-4pt}w(t)\hspace{-6pt}
\sup_{s\in(t,\infty)}\!\!
\left(\int_{t}^{s}\!u(s)\ds\!\right)^{\!\frac{r}{m}}
\!\!\left(\int_{s}^{\infty} \hspace{-6pt} v(s)^{1-p'}\!\ds\!\right)^{\!\frac{r}{p'}} \hspace{-6pt} \dt\!\right)^{\!\frac{1}{r}}\!\!.
\]
Consequently,
\Bdef{5}\Bdef{6}\Bdef{7}
\[
A_2
\approx
\B{5}+\B{6}+\B{7},
\]
where
\[
\B{5}
=\!\!
\left(\sum_{k\in\K}\int_{\Delta_{k-1}}
\hspace{-6pt}\left(\int_{0}^{t} \! w(z)\dz\!\right)^{\!\frac{r}{p}}
\hspace{-6pt} w(t)
\hspace{-4pt}\sup_{s\in(t,t_k)} \!
\left(\int_{t}^{s} \! u(y)\dy\! \right)^{\!\!\frac{r}{m}}
\hspace{-6pt} \left(\int_{s}^{t_k} \hspace{-6pt} v(y)^{1-p'}\!\dy\right)^{\!\!\frac{r}{p'}}\hspace{-6pt} \dt\!\right)^{\!\!\frac{1}{r}}\!\!,
\]
\[
\B{6}
=\!\!
\left(\sum_{k\in\K\setminus\{K\}}\int_{\Delta_{k-1}}
\!\!\!\left(\int_{0}^{t} \! w(z)\dz\!\right)^{\!\!\frac{r}{p}}
\!w(t)\!
\left(\int_{t}^{t_k}\!\!\!u(s)\ds\!\right)^{\!\!\frac{r}{m}} \!\!\!\dt
\left(\int_{t_k}^{\infty}\!\!\!\!v(s)^{1-p'}\!\ds\!\right)^{\!\!\frac{r}{p'}}\!\right)^{\!\!\!\frac{1}{r}}\!\!,
\]
and
\[
\B{7}\!
=\!\!
\left(\sum_{k\in\K\setminus\{K\}}\int_{\!\Delta_{k\!-\!1}}
\!\!\!\left(\int_{0}^{t}\!w(z)\dz\!\right)^{\!\!\frac{r}{p}}
\!\!w(t)\dt
\!\!\sup_{s\in(t_k,\infty)}\!\!
\left(\int_{t_k}^{s}\!\!u(y)\dy\right)^{\!\!\frac{r}{m}}\!\!\!
\left(\int_{s}^{\infty}\!\!\!\!v(y)^{1-p'}\!\dy\right)^{\!\!\frac{r}{p'}}\!\right)^{\!\!\frac{1}{r}}\!\!.
\]
We shall now estimate each of these three terms separately. One has, due to the definition of the sets $\K_1$ and $\K_2$, related identities \eqref{E:copson-7}, \eqref{E:copson-8}, and the readily verified identity $\frac{r}{m}=\frac{rq}{mp}+\frac{q}{m}$, the following:
\begin{align*}
\B{5}
&\approx
\left(\sum_{k\in\K_1}\cdots\right)^{\frac{1}{r}}
+
\left(\sum_{k\in\K_2}\cdots\right)^{\frac{1}{r}}\\
&\lesssim
\left(\sum_{k\in\K_1}
\left(\int_{0}^{t_k}w(z)\dz\right)^{\frac{r}{q}}
\sup_{s\in \Delta_{k-1}}
\left(\int_{t_{k-1}}^{s} \hspace{-6pt} u(y)\dy\right)^{\!\!\frac{r}{m}}
\hspace{-6pt} \left(\int_{s}^{t_k}v(y)^{1-p'}\dy\right)^{\frac{r}{p'}}\right)^{\frac{1}{r}}\\
&\quad+
\left(\sum_{k\in\K_2}\int_{\Delta_{k-1}}\hspace{-6pt}
\left(\int_{0}^{t}\!w(z)\dz\!\right)^{\!\!\frac{r}{p}}
w(t)
\left(\int_{t}^{t_k}\hspace{-6pt}u(y)\dy\!\right)^{\!\!\frac{r}{m}}\!\!
\left(\int_{\Delta_{k-1}}\hspace{-14pt}v(y)^{1-p'}\!\dy\!\right)^{\!\!\frac{r}{p'}}\hspace{-8pt}\dt\!\right)^{\!\!\frac{1}{r}}\\
&\lesssim
\left(\sum_{k\in\K_1}
\left(\int_{0}^{t_k}w(z)\dz\right)^{\!\!\frac{r}{q}}
\sup_{s\in \Delta_{k-1}}
\left(\int_{t_{k-1}}^{s}\!\!\!u(y)\dy\right)^{\!\!\frac{r}{m}}
\left(\int_{s}^{t_k}\!\!v(y)^{1-p'}\dy\right)^{\frac{r}{p'}}\right)^{\!\!\frac{1}{r}}\\
&\quad +
\left(\sum_{k\in\K_2}
\vp(t_k)^r
\left(\int_{\Delta_{k-1}}
v(t)^{1-p'}\dt\right)^{\frac{r}{p'}}\right)^{\frac{1}{r}}\\
&\lesssim
\left(\sum_{k\in\K_1}
\left(\int_{0}^{t_{k-1}}\!\!\!\!w(z)\dz\right)^{\!\!\frac{r}{q}}
\sup_{s\in \Delta_{k-1}}
\left(\int_{t_{k-1}}^{s}\!\!u(y)\dy\right)^{\!\!\frac{r}{m}}\!\!
\left(\int_{s}^{t_k}\!\!v(y)^{1-p'}\dy\right)^{\!\!\frac{r}{p'}}\!\right)^{\!\!\frac{1}{r}}\\
&\quad +
\left(\sum_{k\in\K_2}
\vp(t_{k-1})^r
\left(\int_{\Delta_{k-1}}
v(t)^{1-p'}\dt\right)^{\frac{r}{p'}}\right)^{\frac{1}{r}},
\end{align*}
hence
\[
B_7\lesssim D_2.
\]
Similarly, making use of the fact that $\frac{r}{q}=\frac{r}{p}+1$, we get
\begin{align*}
\B{6}
&=\!
\left(\sum_{k\in\K\setminus\{K\}}\!\int_{\Delta_{k\!-\!1}}\!\!\!
\left(\int_{0}^{t}\!w(z)\dz\!\right)^{\!\!\frac{r}{p}}
\!\!w(t)
\left(\int_{t}^{t_k}\!\!\!u(s)\ds\!\right)^{\!\!\frac{r}{m}}\!\!
\left(\int_{t_k}^{\infty}\!\!\!v(s)^{1-p'}\!\ds\!\right)^{\!\!\frac{r}{p'}}\!\!\!\dt\!\right)^{\!\!\frac{1}{r}}\\
&=\!
\left(\sum_{k\in\K\setminus\{K\}}\int_{\Delta_{k-1}}
\left(\int_{0}^{t}w(z)\dz\right)^{\frac{r}{p}}
\left(\int_{t}^{t_k}u(s)\ds\right)^{\frac{rq}{mp}}
w(t) \right.\\
&\qquad \left.\times
\left(\int_{t}^{t_k}u(s)\ds\right)^{\frac{q}{m}}
\left(\int_{t_k}^{\infty}v(s)^{1-p'}\ds\right)^{\frac{r}{p'}}\dt\right)^{\frac{1}{r}}\\
&\lesssim\!
\left(\sum_{k\in\K\setminus\{K\}}
\vp(t_k)^{r}
\left(\int_{t_k}^{\infty}v(s)^{1-p'}\ds\right)^{\frac{r}{p'}}\right)^{\frac{1}{r}}.
\end{align*}
By \eqref{E:copson-6} and Lemma \ref{L:20}, this yields
  \[
    \B{6} \lesssim \left(\sum_{k\in\K\setminus\{K\}} \vp(t_k)^{r} \left(\int_{\Delta_k}v(s)^{1-p'}\ds\right)^{\frac{r}{p'}}\right)^{\frac{1}{r}} \lesssim D_2.
  \]
Returning to $\B{7}$, we have
\begin{align*}
\B{7}
&=
\left(\sum_{k\in\K\setminus\{K\}}\int_{\Delta_{k-1}}
\left(\int_{0}^{t}w(z)\dz\right)^{\frac{r}{p}}
w(t)\dt \right.\\
&\quad\qquad\left.\times
\sup_{k\leq j\leq K-1}\sup_{s\in \Delta_j}
\left(\int_{t_k}^{s}u(y)\dy\right)^{\frac{r}{m}}
\left(\int_{s}^{\infty}v(y)^{1-p'}\dy\right)^{\frac{r}{p'}}\right)^{\frac{1}{r}}\\
&\lesssim\!
\left(\sum_{k\in\K\setminus\{K\}} \hspace{-6pt}
\left(\int_{0}^{t_k} \hspace{-6pt} w(z)\dz\!\right)^{\!\!\frac{r}{q}}
\hspace{-6pt}\sup_{k\leq j\leq K\!-\!1}\sup_{s\in \Delta_j}
\left(\int_{t_j}^{s}u(y)\dy\right)^{\!\!\frac{r}{m}}\hspace{-6pt}
\left(\int_{s}^{\infty}\hspace{-6pt}v(y)^{1-p'}\dy\right)^{\!\!\frac{r}{p'}}\!\right)^{\!\!\frac{1}{r}}\\
&\quad +
\left(\sum_{k\in\K\setminus\{K-1,K\}}
\left(\int_{0}^{t_k}w(z)\dz\right)^{\frac{r}{q}} \right.\\
& \quad\qquad \times\left.
\sup_{k+1\leq j\leq K-1}
\left(\int_{t_k}^{t_j}u(y)\dy\right)^{\frac{r}{m}}
\left(\int_{t_j}^{\infty}v(y)^{1-p'}\dy\right)^{\frac{r}{p'}}\right)^{\frac{1}{r}}.
\end{align*}
By \eqref{E:copson-5}, \eqref{E:copson-6} and Lemma \ref{L:20}, we can continue as follows.
  \begin{align*}
    \B{7}
      &\lesssim
        \left(\sum_{k\in\K\setminus\{K\}}
        \left(\int_{0}^{t_k}w(z)\dz\right)^{\!\frac{r}{q}}
        \sup_{s\in \Delta_k}
        \left(\int_{t_k}^{s}u(y)\dy\right)^{\!\frac{r}{m}}\!
        \left(\int_{s}^{\infty}v(y)^{1-p'}\dy\right)^{\!\frac{r}{p'}}\right)^{\!\frac{1}{r}}\\
      &\quad +
        \left(\sum_{k\in\K\setminus\{K\!-\!1,K\}}\hspace{-6pt}
        \left(\int_{0}^{t_k} \hspace{-6pt} w(z)\dz\!\right)^{\!\!\frac{r}{q}}
        \!\sum_{j=k+1}^{K-1}\!\!
        \left(\int_{t_k}^{t_j}u(y)\dy\!\right)^{\!\!\frac{r}{m}}\!\!
        \left(\int_{t_{j}}^{\infty} \hspace{-6pt} v(y)^{1-p'}\!\dy\!\right)^{\!\frac{r}{p'}}\!\right)^{\!\!\frac{1}{r}} \\
      &\lesssim
        \left(\sum_{k\in\K\setminus\{K\}}\!\!
        \left(\int_{0}^{t_k}\!\!w(z)\dz\right)^{\!\!\frac{r}{q}}
        \sup_{s\in \Delta_k}
        \left(\int_{t_k}^{s}u(y)\dy\right)^{\!\!\frac{r}{m}}
        \left(\int_{s}^{t_{k+1}}\!\!\!v(y)^{1-p'}\dy\right)^{\!\!\frac{r}{p'}}\!\right)^{\!\!\frac{1}{r}}\\
      &\quad +
        \left(\sum_{k\in\K\setminus\{K\!-\!1,K\}}\!\!
        \left(\int_{0}^{t_k}\!w(z)\dz\!\right)^{\!\!\frac{r}{p}}
        \left(\int_{\Delta_k}u(y)\dy\right)^{\!\!\frac{r}{m}}\!\!
        \left(\int_{t_{k+1}}^{\infty}\!\!\!v(y)^{1-p'}\dy\right)^{\!\!\frac{r}{p'}}\!\right)^{\!\!\frac{1}{r}}\\
      &\quad +
        \left(\sum_{k\in\K\setminus\{K\!-\!1,K\}}\!\!\!
        \left(\int_{0}^{t_k}\!\!w(z)\dz\!\right)^{\!\!\frac{r}{q}}
        \!\sum_{j=k+1}^{K\!-\!1}\!\!
        \left(\int_{t_k}^{t_j}\!u(y)\dy\!\right)^{\!\!\frac{r}{m}}\!\!
        \left(\int_{t_{j}}^{\infty}\!\!v(y)^{1-p'}\!\dy\!\right)^{\!\!\frac{r}{p'}}\!\right)^{\!\!\frac{1}{r}}\!\!,
     \end{align*}
     that is,
  \begin{align*}
    \B{7}
      &\lesssim
        \left(\sum_{k\in\K}
        \sup_{s\in\Delta_{k-1}}\vp(s)^r
        \left(\int_{s}^{t_k}v(y)^{1-p'}\dy\right)^{\frac{r}{p'}}\right)^{\frac{1}{r}}\\
      &\quad +
        \left(\sum_{j\in\K\setminus\{K\}}\!\!
        \left(\int_{t_{j}}^{\infty}\!\!v(y)^{1-p'}\!\dy\right)^{\!\!\frac{r}{p'}}
        \!\!\sum_{k=-\infty}^{j-1}\!\!
        \left(\int_{0}^{t_k}\!w(z)\dz\right)^{\!\!\frac{r}{q}}
        \left(\int_{t_k}^{t_j\!}u(y)\dy\right)^{\!\!\frac{r}{m}}\!\right)^{\!\!\frac{1}{r}}\!\!.
  \end{align*}
Observe that $\frac{r}{q}>1$ and recall the inequality \eqref{E:copson-6}. It allows us to treat the inner sum in the second term as follows.
\begin{align*}
\B{7}\!
&\lesssim\!
D_2\!+\!
\left(\sum_{j\in\K\setminus \{K\}}\!\!\!
\left(\int_{t_{j}}^{\infty}\!\!\!\!v(y)^{1-p'}\!\dy\!\right)^{\!\!\frac{r}{p'}}\!\!\!
\left(\sum_{k=-\infty}^{j-1}\!\!
\left(\int_{\Delta_{k-1}}\hspace{-14pt}w(s)\ds\!\right)\!
\left(\int_{t_k}^{t_j}\!\!u(y)\dy\!\right)^{\!\!\frac{q}{m}}\!\right)^{\!\!\frac{r}{q}}\!\right)^{\!\!\frac{1}{r}}\\
&\lesssim
D_2+
\left(\sum_{j\in\K\setminus \{K\}}
\vp(t_j)^r
\left(\int_{t_{j}}^{\infty}v(y)^{1-p'}\dy\right)^{\frac{r}{p'}}\right)^{\frac{1}{r}}\\
&\lesssim
D_2.
\end{align*}
The last step follows from \eqref{E:copson-6} and Lemma \ref{L:20}. So far we have established the inequality
	$$A_2 \lesssim D_2.$$
In the next step, we will find a~similar estimate for $A_3$. Observe that the assumption $p\le m$ will not be used during the process, therefore the following estimates are valid even for $0<m<p$,\ $p>1$. We have
\Bdef{8}\Bdef{9}\Bdef{10}
\begin{align*}
A_3
&=
\left(\sum_{k\in\K}
\int_{\Delta_{k-1}}
\!\!\vp(t)^{\frac{rq}{p}}w(t)
\sup_{s\in(t,\infty)}
\left(\int_t^s u(y)\dy\right)^{\frac qm}\!
\left(\int_s^{\infty}v(y)^{1-p'}\dy\right)^{\frac r{p'}}\!\!\dt\right)^{\frac 1r}\\
&\approx
\left(\sum_{k\in\K}
\int_{\Delta_{k-1}}
\!\!\vp(t)^{\frac{rq}{p}}w(t)
\sup_{s\in(t,t_k)}
\left(\int_t^s u(y)\dy\right)^{\!\frac qm}
\left(\int_s^{t_k}v(y)^{1-p'}\dy\right)^{\!\frac r{p'}}\!\!\dt\right)^{\!\frac 1r}\\
&\quad +\!
\left(\sum_{k\in\K\setminus\{K\}}
\int_{\Delta_{k-1}}
\!\!\vp(t)^{\frac{rq}{p}}w(t)
\left(\int_t^{t_k} u(y)\dy\right)^{\frac qm}\!
\left(\int_{t_k}^{\infty}v(y)^{1-p'}\dy\right)^{\frac r{p'}}\!\!\dt\right)^{\!\frac 1r}\\
&\quad+\!
\left(\sum_{k\in\K\setminus\{K\}}\!
\int_{\Delta_{k-1}}
\hspace{-14pt}\vp(t)^{\frac{rq}{p}}w(t)\dt
\!\!\sup_{s\in(t_k,\infty)}\!\!
\left(\int_{t_k}^{s} u(y)\dy\right)^{\!\!\frac qm}\!\!
\left(\int_{s}^{\infty}\!\!v(y)^{1-p'}\!\dy\right)^{\!\!\frac r{p'}}\!\right)^{\!\!\frac 1r}\\
&= \B{8}+\B{9}+\B{10},
\end{align*}
say. By \eqref{E:copson-7} and \eqref{E:copson-8}, one gets
  \begin{align*}
      \B{8}
        &\approx
          \left(\sum_{k\in\K_1}\cdots\right)^{\frac{1}{r}}+
          \left(\sum_{k\in\K_2}\cdots\right)^{\frac{1}{r}}\\
        &\leq
          \left(\sum_{k\in\K_1}
          \int_{\Delta_{k-1}}\!\!\!\!w(t)\dt
          \sup_{s\in\Delta_{k-1}}
          \vp(s)^{\frac{rq}{p}}
          \left(\int_{t_{k-1}}^{s} \!\!u(y)\dy\right)^{\!\!\frac qm}\!\!
          \left(\int_{s}^{t_k}\!\!v(y)^{1-p'}\dy\right)^{\!\!\frac r{p'}}\right)^{\!\!\frac 1r}\\
        &\qquad+
          \left(\sum_{k\in\K_2}
          \vp(t_k)^{r}
          \left(\int_{\Delta_{k-1}}
          v(y)^{1-p'}\dy\right)^{\frac r{p'}}\right)^{\frac 1r}\\
        &\lesssim
          \left(\sum_{k\in\K_1}
          \int_{\Delta_{k-2}}\!\!\!\!w(t)\dt
          \sup_{s\in\Delta_{k-1}}\!
          \vp(s)^{\frac{rq}{p}}
          \left(\int_{t_{k-1}}^{s} \!\!\!u(y)\dy\right)^{\!\!\frac qm}\!\!
          \left(\int_{s}^{t_k}v(y)^{1-p'}\dy\right)^{\!\!\frac r{p'}}\!\right)^{\!\!\frac 1r}\\
        &\qquad +
          \left(\sum_{k\in\K_2}
          \vp(t_{k-1})^{r}
          \left(\int_{\Delta_{k-1}}
          v(y)^{1-p'}\dy\right)^{\frac r{p'}}\right)^{\frac 1r}\\
        &\lesssim
          D_2.
  \end{align*}
By the standard argument based on \eqref{E:copson-6} and Lemma \ref{L:20}, we also get
  \[
    \B{9}
      \le
      \left(\sum_{k\in\K\setminus\{K\}}
      \vp(t_k)^{r}
      \left(\int_{t_k}^{\infty}v(y)^{1-p'}\dy\right)^{\frac r{p'}}\right)^{\frac 1r}\lesssim D_2.
  \]
Moreover, the same argument is used to show the following.
\begin{align*}
\B{10}
&=\!
\left(\sum_{k\in\K\setminus\{K\}}
\int_{\Delta_{k-1}}\!
\!\!\!\!\!\vp(t)^{\frac{rq}{p}}w(t)\dt\!\!
\sup_{s\in(t_k,\infty)}\!\!
\left(\int_{t_k}^{s} \!\!u(y)\dy\!\right)^{\!\!\frac qm}\!\!
\left(\int_{s}^{\infty}\!\!\!\!v(y)^{1-p'}\!\dy\!\right)^{\!\!\frac r{p'}}\!\right)^{\!\!\frac 1r}\\
&\leq\!
\left(\sum_{k\in\K\setminus\{K\}}
\vp(t_k)^{\frac{rq}{p}}
\sup_{s\in(t_k,\infty)}
\vp(s)^{q}
\left(\int_{s}^{\infty}v(y)^{1-p'}\dy\right)^{\frac r{p'}}\right)^{\frac 1r}\\
&\lesssim\!
\left(\sum_{k\in\K\setminus\{K\}}
\vp(t_k)^{\frac{rq}{p}}
\sup_{s\in\Delta_{k}}
\vp(s)^{q}
\left(\int_{s}^{\infty}v(y)^{1-p'}\dy\right)^{\frac r{p'}}\right)^{\frac 1r}\\
&\leq\!
\left(\sum_{k\in\K}
\sup_{s\in\Delta_{k-1}}
\vp(s)^{r}
\left(\int_{s}^{\infty}v(y)^{1-p'}\dy\right)^{\frac r{p'}}\right)^{\frac 1r}\\
&\lesssim D_2.
\end{align*}
Hence, we finally obtain
\[
A_3\lesssim D_2.
\]
It remains to estimate $D_2$. We shall use the saturation of the discrete H\"older inequality (see Lemma \ref{L:40}) in the following form: there exists a~sequence of non-negative numbers, $\{a_k\}_{k\in\K}$, such that $\sum_{k\in\K}a_k^p=1$ and
  \begin{align*}
    D_2&=
    \left(\sum_{k\in\K}\sup_{t\in\Delta_{k-1}}\vp(t)^r
    \left(\int_t^{t_k}v(s)^{1-p'}\ds\right)^{\frac{r}{p'}}
    \right)^{\frac{1}{r}}\\
    &=
    \left(\sum_{k\in\K} a_k^q
    \sup_{t\in\Delta_{k-1}}\vp(t)^q
    \left(\int_t^{t_k}v(s)^{1-p'}\ds\right)^{\frac{q}{p'}}
    \right)^{\frac{1}{q}}.
  \end{align*}
We define the functions
\[
g(t)=\sum_{k\in\K}a_kg_k(t),\qquad h(t)=\sum_{k\in\K}a_kh_k(t)
\]
for any $t>0$, and note that $\|g\|_{L^p(v)}=\|h\|_{L^p(v)}=1$. The functions $g_k$, $h_k$ were defined earlier, see \eqref{E:nutnost-2} and \eqref{E:nutnost-1}. (The definition of $h_k$ is independent of $q$.) Integrating by parts, we get
\begin{align*}
D_2 &=
    \left(\sum_{k\in\K} a_k^q
    \sup_{t\in\Delta_{k-1}}\vp(t)^q
    \left(\int_t^{t_k}v(s)^{1-p'}\ds\right)^{\frac{q}{p'}}
    \right)^{\frac{1}{q}} \\
&\approx
\left(\sum_{k\in\K}a_k^{q}\vp(t_{k-1})^{q}
\left(\int_{\Delta_{k-1}}v(s)^{1-p'}\dy\right)^{\frac{q}{p'}}\right)^{\frac{1}{q}}\\
&\qquad +
\left(\sum_{k\in\K}a_k^{q}\sup_{t\in\Delta_{k-1}}
\left(\int_{t_{k-1}}^{t}\vp(s)^{m-1}\vp'(s)\ds\right)^{\frac{q}{m}}
\left(\int_{t}^{t_k}v(s)^{1-p'}\dy\right)^{\frac{q}{p'}}\right)^{\frac{1}{q}}\\
&\lesssim
\left(\sum_{k\in\K}a_k^{q}\vp(t_{k-1})^{q}
\left(\int_{\Delta_{k-1}}g_k(s)\dy\right)^{q}\right)^{\frac{1}{q}}\\
&\qquad +
\left(\sum_{k\in\K}a_k^{q}
\left(\int_{\Delta_{k-1}}\vp(s)^{m-1}\vp'(s)
\left(\int_{s}^{t_k}h_k(y)\dy\right)^{m}\ds\right)^{\frac{q}{m}}
\right)^{\frac{1}{q}}.
\end{align*}
Using this and the properties of $g$ and $h$, we arrive at
\begin{align*}
D_2&\lesssim
\left(\sum_{k\in\K}\vp(t_{k-1})^{q}
\left(\int_{\Delta_{k-1}}(g(s)+h(s))\ds\right)^{q}\right)^{\frac{1}{q}}\\
&\qquad +
\left(\sum_{k\in\K}
\left(\int_{\Delta_{k-1}}\vp(s)^{m-1}\vp'(s)
\left(\int_{s}^{t_k}(g(y)+h(y))\dy\right)^{m}\ds\right)^{\frac{q}{m}}
\right)^{\frac{1}{q}}.
\end{align*}
Finally, using \eqref{E:diskr}, we obtain
  \[
    D_2 \lesssim C \|h+g\|_{L^p(v)} \lesssim C,
  \]
since $\|h+g\|_{p}\leq 2$. This proves the necessity part of the assertion in the case (b).\\

The proof of the remaining cases relies on the saturation of the Hardy inequality in the case $0<m<p<\infty$ and $p>1$ (see Lemma \ref{L:35}(ii)). Precisely, for every $k\in\K$ there exists a function $f_k\in\MM$ such that $\supp f_k\subset\Delta_{k-1}$, $\int_{\Delta_{k-1}}f_k(t)^{p}v(t)\dt=1$ and
	\begin{align}
		&\left(\int_{\Delta_{k-1}}\!\!
		\left(\int_{t_{k-1}}^{t}\!\!\!\vp(s)^{m-1}\vp'(s)\ds\right)^{\!\!\frac{p}{p-m}}\!\!\!
		\lt \int_t^{t_k} \!\!\!v(z)^{1-p'} \dz \rt^{\!\!\!\frac{p(m-1)}{p-m}}
		\!\!\!\!\!\!\!v(t)^{1-p'}\dt\right)^{\!\!\!\frac{p-m}{mp}}\label{E:nutnost-f}\\
		&\qquad \lesssim
		\left(\int_{\Delta_{k-1}}
		\vp(t)^{m-1}\vp'(t)
		\left(\int_{t}^{t_k}f_k(s)\ds\right)^{m}\dt\right)^{\frac{1}{m}}.\nonumber
	\end{align}

(c) Let $0<m<p<\infty$ and $1<p\leq q$. By the universal estimate proven in the beginning of the necessity part, we already know that $A_1\le C<\infty$ is necessarily true. Therefore, the implication \eqref{E:psourtek} keeps $A_4$ meaningful even in case that $m<1$ (which is allowed here). Note that
\[
D_1\lesssim D_3,
\]
because
\begin{align*}
D_1
&\approx
\sup_{k\in\K}
\sup_{t\in\Delta_{k-1}}\vp(t)
\left(\int_{t}^{t_k}
\left(\int_{s}^{t_k}v(y)^{1-p'}\dy\right)^{\frac{p(m-1)}{p-m}}
v(s)^{1-p'}\ds\right)^{\frac{p-m}{pm}}\\
&\leq
\sup_{k\in\K}
\left(\int_{\Delta_{k-1}}
\vp(t)^{\frac{mp}{p-m}}
\left(\int_{t}^{t_k}v(s)^{1-p'}\ds\right)^{\frac{p(m-1)}{p-m}}
v(t)^{1-p'}\dt\right)^{\frac{p-m}{pm}}\\
&=D_3.
\end{align*}

We begin with estimating $A_4$. Taking \eqref{E:psourtek} into account and integrating by parts, for every fixed $t>0$ we get
\begin{align*}
	&\int_t^{\infty}\left(\int_t^s u(y)\dy\right)^{\frac m{p-m}}u(s)\left(\int_{s}^{\infty}v(\tau)^{1-p'}\dtau\right)^{\frac{m(p-1)}{p-m}}\ds \\
	&\approx \int_t^{\infty} \left(\int_t^s u(y)\dy\right)^{\frac p{p-m}}\left(\int_{s}^{\infty}v(\tau)^{1-p'}\dtau\right)^{\frac{p(m-1)}{p-m}}v(s)^{1-p'}\ds \\
	&\qquad + \lim_{s\to\infty} \left(\int_t^s u(y)\dy\right)^{\frac p{p-m}} \left(\int_{s}^{\infty}v(\tau)^{1-p'}\dtau\right)^{\frac{m(p-1)}{p-m}}\\
	&\approx \int_t^{\infty} \left(\int_t^s u(y)\dy\right)^{\frac p{p-m}} \left(\int_{s}^{\infty}v(\tau)^{1-p'}\dtau\right)^{\frac{p(m-1)}{p-m}}v(s)^{1-p'}\ds.
\end{align*}
This verifies that $A_4\approx A_4^*$, hence we have
\Bdef{11}\Bdef{12}\Bdef{13}
\begin{align*}
A_4
\approx A_4^*
&=
\sup_{k\in\K}
\sup_{t\in\Delta_{k-1}}
\left(\int_0^tw(s)\ds\right)^{\frac1q} \\
&\qquad\times
\left(\int_t^{\infty}\!\!\left(\int_t^s u(y)\dy\right)^{\!\!\frac m{p-m}} \!\!u(s) \left(\int_{s}^{\infty}v(\tau)^{1-p'}\dtau\right)^{\!\!\frac{m(p-1)}{p-m}} \!\!\!\ds\right)^{\!\!\frac{p-m}{pm}}\\
&\approx
\B{11}+\B{12}+\B{13},
\end{align*}
where
\[
\B{11}\!
=\!
\sup_{k\in\K}
\sup_{t\in\Delta_{k\!-\!1}}\!\!
\left(\int_0^t\!\!w(s)\ds\!\right)^{\!\!\frac1q}\hspace{-4pt}
\left(\int_t^{t_k}\hspace{-6pt}\left(\int_t^s \!\!u(y)\!\dy\!\right)^{\!\!\frac m{p\!-\!m}} \!\!\!\!u(s) \!\left(\int_{s}^{t_k} \!\!\!v(z)^{1\!-\!p'}\!\dz\!\right)^{\!\!\!\!\frac{m(p-1)}{p-m}} \hspace{-14pt}\ds\!\right)^{\hspace{-7pt}\frac{p-m}{pm}}\hspace{-8pt},
\]
\[
\B{12}
=
\sup_{k\in\K\setminus\{K\}}
\sup_{t\in\Delta_{k-1}}
\left(\int_0^tw(s)\ds\right)^{\frac1q}
\left(\int_{t}^{t_k}u(s)\ds\right)^{\frac{1}{m}}\left(\int_{t_k}^{\infty}v(z)^{1-p'}\dz\right)^{\frac{1}{p'}}
\]
and
	\begin{align*}
		\B{13} &
		=
		\sup_{k\in\K\setminus\{K\}}
		\sup_{t\in\Delta_{k-1}}
		\left(\int_0^{t} w(s)\ds\right)^{\frac1q}\\
		& \qquad
		\left(\int_{t_k}^{\infty}\left(\int_{t}^s u(y)\dy\right)^{\frac m{p-m}} u(s) \left(\int_{s}^{\infty}v(z)^{1-p'}\dz\right)^{\frac{m(p-1)}{p-m}} \ds\right)^{\frac{p-m}{pm}}.
	\end{align*}
Integrating by parts, we obtain
	\begin{align*}
		\B{11}&\lesssim \sup_{k\in\K}
		\sup_{t\in\Delta_{k-1}}
		\left(\int_0^t w(s)\ds\right)^{\frac1q}\\
		& \qquad\times
		\left(\int_t^{t_k}\hspace{-4pt}\left(\int_t^s u(y)\dy\right)^{\frac p{p-m}} \hspace{-4pt}\left(\int_{s}^{t_k} v(z)^{1-p'}\dz\right)^{\hspace{-4pt}\frac{p(m-1)}{p-m}} \hspace{-6pt} v(s)^{1-p'}\ds\right)^{\hspace{-6pt}\frac{p-m}{pm}}\\
		&\lesssim D_3.
	\end{align*}
Similarly, we get
\[
\B{12}\lesssim D_1\lesssim D_3.
\]
Now we are going to estimate the term $\B{13}$. First, by \eqref{E:copson-5} and Lemma \ref{L:20} we get
\Bdef{14}\Bdef{15}
\begin{align*}
\B{13}
&=
\sup_{k\in\K\setminus\{K\}}
\sup_{t\in\Delta_{k-1}}
\left(\int_0^{t}w(s)\ds\right)^{\frac{m}{pq}+\frac{p-m}{pq}} \\
&\qquad\qquad\times
\left(\int_{t_k}^{\infty}\left(\int_{t}^s u(y)\dy\right)^{\frac{m}{p-m}}u(s)\left(\int_{s}^{\infty}v(z)^{1-p'}\dz\right)^{\!\!\frac{m(p-1)}{p-m}}\!\!\!\ds\right)^{\!\!\frac{p-m}{pm}}\\
&\leq
\sup_{k\in\K\setminus\{K\}}
\!\!\left(\int_0^{t_k}\!\!w(s)\ds\right)^{\!\!\!\frac{p-m}{pq}}\!\!\!
\left(
\int_{t_{k}}^{\infty}\!\!\!\vp(s)^{\frac{m^2}{p-m}}u(s)
\left(\int_{s}^{\infty}\!\!\!v(z)^{1-p'}\dz\right)^{\hspace{-6pt}\frac{m(p-1)}{p-m}}\hspace{-10pt}\ds
\right)^{\hspace{-6pt}\frac{p-m}{pm}}\\
&\lesssim
\sup_{k\in\K\setminus\{K\}}\hspace{-4pt}
\left(\int_0^{t_k}\hspace{-4pt}w(s)\ds\right)^{\!\!\frac{p-m}{pq}}
\hspace{-6pt}\left(
\int_{\Delta_{k}}\hspace{-6pt}\vp(s)^{\frac{m^2}{p-m}}u(s)
\left(\int_{s}^{\infty}\hspace{-6pt}v(z)^{1-p'}\dz\right)^{\hspace{-6pt}\frac{m(p-1)}{p-m}}\hspace{-10pt}\ds
\right)^{\hspace{-6pt}\frac{p-m}{pm}}\\
&\approx
\sup_{k\in\K\setminus\{K\}}\hspace{-4pt}
\left(\int_0^{t_k}\hspace{-4pt}w(s)\ds\right)^{\!\!\frac{p-m}{pq}}
\hspace{-6pt}\left(
\int_{\Delta_{k}}\hspace{-6pt}\vp(s)^{\frac{m^2}{p-m}}u(s)
\left(\int_{s}^{t_{k+1}}\hspace{-10pt}v(z)^{1-p'}\dz\right)^{\hspace{-6pt}\frac{m(p-1)}{p-m}}\hspace{-10pt}\ds
\right)^{\hspace{-6pt}\frac{p-m}{pm}}\\
&\qquad +
\sup_{k\in\K\setminus\{K,K-1\}}
\left(\int_0^{t_k}w(s)\ds\right)^{\!\!\frac{p-m}{pq}}
\left(
\int_{\Delta_{k}}\vp(s)^{\frac{m^2}{p-m}}u(s)
\ds
\right)^{\frac{p-m}{pm}}\\
&\hspace{120pt} \times
\left(\int_{t_{k+1}}^{\infty}v(z)^{1-p'}\dz\right)^{\!\!\frac{1}{p'}}\\
&=\B{14}+\B{15},
\end{align*}
say. We note that for every $a\in\Delta_k$, integration by parts gives
\begin{multline}\label{E:Hardy-lemma}
\int_{a}^{t_{k+1}}
u(s)
\left(\int_{s}^{t_{k+1}} \hspace{-5pt} v(z)^{1-p'}\dz\right)^{\!\frac{m(p-1)}{p-m}}\hspace{-10pt} \ds \\
\lesssim
\int_{a}^{t_{k+1}}
\left(\int_{t_k}^{s}u(y)\dy\right)
\left(\int_{s}^{t_{k+1}} \hspace{-5pt} v(z)^{1-p'}\dz\right)^{\!\frac{p(m-1)}{p-m}}
\hspace{-10pt} v(s)^{1-p'}\ds.
\end{multline}
Therefore, since the function $s\mapsto \vp(s)^{\frac{m^2}{p-m}}$ is increasing, from the Hardy lemma \cite[p.~56]{BS} we get
\begin{align*}
\B{14}
&\lesssim
\sup_{k\in\K\setminus\{K\}}
\left(\int_0^{t_k}w(s)\ds\right)^{\frac{p-m}{pq}} \\
& \hspace{50pt} \times\!
\left(
\int_{\Delta_{k}}\!\!\vp(s)^{\frac{m^2}{p-m}}
\!\!\int_{t_k}^{s}\!u(y)\dy
\left(
\int_{s}^{t_{k+1}}\hspace{-12pt}v(z)^{1-p'}\dz\right)^{\hspace{-4pt}\frac{p(m-1)}{p-m}}
\hspace{-12pt}v(s)^{1-p'}\!\ds
\right)^{\hspace{-4pt}\frac{p-m}{pm}}\\
&\leq D_3,
\end{align*}
since $\frac{m^2}{p-m}+m=\frac{mp}{p-m}$. Next, using \eqref{E:copson-6} and Lemma \ref{L:20} we obtain
\begin{align*}
\B{15}
&\lesssim
\!\!\sup_{k\in\K\setminus\{K,K\!-\!1\}}\!\!
\left(\int_0^{t_k}\!\!\!w(s)\ds\!\right)^{\!\!\!\frac{p-m}{pq}}
\!\!\!\left(
\int_{\Delta_{k}}\!\!\!u(y)\!\dy\!
\right)^{\!\frac{p-m}{pm}}
\!\!\!\!\vp(t_{k+1})^{\frac{m}{p}}\!\!
\left(\int_{t_{k+1}}^{\infty}\!\!\!v(z)^{1-p'}\dz\!\right)^{\!\!\frac{1}{p'}}\\
&\lesssim
\sup_{k\in\K\setminus\{K,K-1\}}
\vp(t_{k+1})
\left(\int_{t_{k+1}}^{\infty}v(z)^{1-p'}\dz\right)^{\frac{1}{p'}}\\
&\lesssim
\sup_{k\in\K}
\vp(t_{k-1})
\left(\int_{\Delta_{k-1}}v(z)^{1-p'}\dz\right)^{\frac{1}{p'}}\\
&\lesssim D_1\lesssim D_3.
\end{align*}
To finish this case, we need an estimate of $D_3$. Fix $k\in\K$. Then, integrating by parts and using the properties of the functions $g_k$ and $f_k$ (see \eqref{E:nutnost-2} and \eqref{E:nutnost-f}), we get
\begin{align*}
&\left(
\int_{\Delta_{k-1}}
\vp(t)^{\frac{mp}{p-m}}
\left(\int_{t}^{t_k}v(s)^{1-p'}\ds\right)^{\frac{p(m-1)}{p-m}}
v(t)^{1-p'}\dt
\right)^{\frac{p-m}{pm}}\\
&\approx
\vp(t_{k-1})
\left(\int_{\Delta_{k-1}}
v(t)^{1-p'}\dt
\right)^{\frac{1}{p'}}\\
&\quad+
\left(\int_{\Delta_{k-1}}\!\!
\left(\int_{t_{k-1}}^{t}\!\!\vp(s)^{m-1}\vp'(s)\ds\right)^{\!\!\frac{p}{p-m}}\!\!
\left(\int_{t}^{t_k}\!\!v(s)^{1-p'}\ds\right)^{\!\!\frac{p(m-1)}{p-m}}
\!\!\!v(t)^{1-p'}\dt\right)^{\!\!\frac{p-m}{mp}}\\
&\lesssim
\vp(t_{k-1})
\left(\int_{\Delta_{k-1}}
\!\!g_k(t)\dt\right)
+
\left(\int_{\Delta_{k-1}}
\vp(t)^{m-1}\vp'(t)
\left(\int_{t}^{t_k}f_k(s)\ds\right)^{\!m}
\!\!\dt\right)^{\!\!\frac{1}{m}},\\
&\lesssim
\left(\int_{0}^{\infty}
w(t)\left(\int_t^{\infty}u(s)
\left(\int_s^{\infty}g_k(y)+f_k(y)\dy\right)^{m}\ds\right)^{\frac{q}{m}}
\dt\right)^{\frac{1}{q}}\\
&\lesssim C\|g_k+f_k\|_{L^p(v)}.
\end{align*}
In the last estimate we used \eqref{E:diskr}. Recalling that $\|g_k+f_k\|_{L^p(v)}\leq 2$ and taking supremum over all $k\in\K$, we get
\[
D_3\lesssim C,
\]
proving the necessity part of the assertion in the case (c). \\

(d) Let $0<m<p<\infty$, $1<p$ and $0<q<p<\infty$. As in (c), since we already know that $A_1\le C < \infty$, we recall \eqref{E:psourtek} when treating expressions involving the term $\int_t^\infty v(s)^{1-p'}(s)\ds$ with a~negative exponent. In particular, this also justifies the use of integration by parts to show that $A_5\approx A_5^*$.

Recall that, in the necessity part of (b), we have already proved that $A_3\lesssim D_2$, using only the assumption $0<q<p<\infty$ (with $p>1$) and not the relation between $p$ and $m$. Hence, the estimate $A_3\lesssim D_2$ is valid with the setting of parameters from (d) as well. Since $D_2 \lesssim D_4$ is obviously true, we obtain
	\[
		A_3 \lesssim D_4.
	\]
We continue with the term $A_5$. One has
\Bdef{16}\Bdef{17}\Bdef{18}
	\begin{align*}
	A_5
		& = \Bigg(
			\sum_{k\in\K}\int_{\Delta_{k-1}}
			\left(\int_0^t w(s)\ds\right)^{\frac{r}{p}}
			w(t) \\
		& \qquad \times
			\left(
			\int_{t}^{\infty}
			\!\!\left(\int_{t}^{z}u(s)\ds\right)^{\!\!\frac{p}{p-m}} \hspace{-4pt}
			\left(\int_{z}^{\infty} \hspace{-4pt} v(s)^{1-p'}\ds\right)^{\!\!\!\frac{p(m-1)}{p-m}} \hspace{-10pt} v(z)^{1-p'}\dz
			\right)^{\!\!\!\frac{q(p-m)}{m(p-q)}}
			\hspace{-12pt}\dt
			\Bigg)^{\!\!\frac {1}{r}}\\
		&\approx
			\Bigg(
			\sum_{k\in\K}\int_{\Delta_{k-1}}
			\left(\int_0^t w(s)\ds\right)^{\frac{r}{p}}
			w(t) \\
		&\qquad \times
			\left(
			\int_{t}^{t_k}
			\!\!\left(\int_{t}^{z}u(s)\ds\right)^{\!\!\frac{p}{p-m}} \hspace{-4pt}
			\left(\int_{z}^{\infty} \hspace{-4pt} v(s)^{1-p'}\ds\right)^{\!\!\!\frac{p(m-1)}{p-m}} \hspace{-10pt} v(z)^{1-p'}\dz
			\right)^{\!\!\!\frac{q(p-m)}{m(p-q)}}
			\hspace{-12pt}\dt
			\Bigg)^{\!\!\frac {1}{r}}\\
		&\quad +\!
			\Bigg(
			\sum_{k\in\K\setminus\{\!K\!\}}\!\int_{\Delta_{k\!-\!1}}\!\!\!
			\left(\int_0^t \!w(s)\!\ds\!\right)^{\!\!\frac{r}{p}}\!\!
			w(t)
			\left(
			\int_{t}^{t_k}
			\!\!\!\!u(s)\ds\!
			\right)^{\!\!\frac{r}{m}}\!\!\!
			\dt
			\left(\int_{t_k}^{\infty}\!\!\!v(s)^{1\!-\!p'}\!\ds\!\right)^{\!\!\frac{r}{p'}}
			\!\!\Bigg)^{\!\!\frac {1}{r}}\\
		&\quad +\!
			\Bigg(
			\sum_{k\in\K\setminus\{K\}}\int_{\Delta_{k-1}}
			\left(\int_0^t w(s)\ds\right)^{\frac{r}{p}}
			w(t)\dt \\
		&\qquad\times
			\left(
			\int_{t_k}^{\infty}\!\!
			\left(\int_{t_k}^{z}u(s)\ds
			\right)^{\!\!\frac{p}{p-m}}\!\!
			\left(\int_{z}^{\infty}\!\!v(s)^{1-p'}\ds\right)^{\!\!\frac{p(m-1)}{p-m}}\!\!\!\!\!v(z)^{1-p'}\dz
			\right)^{\!\!\frac{q(p-m)}{m(p-q)}}
			\Bigg)^{\!\!\frac {1}{r}}\\
		&=\B{16}+\B{17}+\B{18},
	\end{align*}
say. The three summands will be now estimated separately. First, by the definition of $\K_1$ and $\K_2$ and using \eqref{E:copson-7} and \eqref{E:copson-8}, one has
\allowdisplaybreaks
	\begin{align*}
		\B{16}
			&\approx
				\left(\sum_{k\in\K_1}\cdots\right)^{\frac{1}{r}}+\left(\sum_{k\in\K_2}\cdots\right)^{\frac{1}{r}}\\
			&\lesssim
				\Bigg(
				\sum_{k\in\K_1}
				\left(\int_0^{t_k} w(s)\ds\right)^{\frac{r}{q}}\\
			&\quad\times\left.\!\!
				\left(\int_{\Delta_{k-1}}\!\!
				\left(
				\int_{t_{k-1}}^{z}\!\!u(s)\ds\right)^{\!\!\frac{p}{p-m}}\!\!\!
				\left(\int_{z}^{\infty}\!\!v(s)^{1-p'}\ds\right)^{\!\!\frac{p(m-1)}{p-m}}\!\!\!\!v(z)^{1-p'}\dz
				\right)^{\!\!\frac{q(p-m)}{m(p-q)}}
				\right)^{\!\!\frac {1}{r}}\\
			&\quad+
				\left(
				\sum_{k\in\K_2}
				\vp(t_k)^r
				\left(\int_{\Delta_{k-1}}
				\left(\int_{z}^{\infty}v(s)^{1-p'}\ds\right)^{\frac{p(m-1)}{p-m}}v(z)^{1-p'}\dz
				\right)^{\frac{q(p-m)}{m(p-q)}}
				\right)^{\frac {1}{r}}\\
			&\lesssim
			\left(\vphantom{\Bigg[^\big[}
			\sum_{k\in\K_1}
			\left(\int_0^{t_{k-1}} w(s)\ds\right)^{\frac{r}{q}}\right.\\
			& \qquad \!\left.\times
			\left(\int_{\Delta_{k-1}}\!\!
			\left(
			\int_{t_{k-1}}^{z}\!\!\!\!u(s)\ds\!\right)^{\!\!\frac{p}{p-m}}\!\!
			\left(\int_{z}^{\infty}\!\!\!v(s)^{1-p'}\ds\!\right)^{\!\!\frac{p(m-1)}{p-m}}\!\!\!\!v(z)^{1-p'}\dz\!
			\right)^{\!\!\frac{q(p-m)}{m(p-q)}}
			\right)^{\!\!\frac {1}{r}}\\
			&\quad+
			\left(
			\sum_{k\in\K_2}
			\vp(t_{k-1})^r
			\left(\int_{\Delta_{k-1}}\!\!
			\left(\int_{z}^{\infty}\!v(s)^{1-p'}\ds\right)^{\!\!\frac{p(m-1)}{p-m}}\!\!\!\!v(z)^{1-p'}\dz
			\right)^{\!\!\frac{q(p-m)}{m(p-q)}}
			\right)^{\!\!\frac {1}{r}}\\
			&\lesssim
				\left(
				\sum_{k\in\K}
				\left(\int_{\Delta_{k-1}}
				\vp(z)^{\frac{mp}{p-m}}
				\left(\int_{z}^{\infty}v(s)^{1-p'}\ds\right)^{\!\!\frac{p(m-1)}{p-m}}v(z)^{1-p'}\dz
				\right)^{\!\!\frac{q(p-m)}{m(p-q)}}
				\right)^{\!\!\frac {1}{r}}\!\!.
	\end{align*}
Now, by a double use of integration by parts, as well as using the standard argument involving \eqref{E:copson-6} and Lemma \ref{L:20}, we get
	\begin{align*}
		\B{16}
			&\lesssim
				\left(
				\sum_{k\in\K}
				\vp(t_{k-1})^r
				\left(\int_{t_{k-1}}^{\infty}v(s)^{1-p'}\ds\right)^{\frac{r}{p'}}\right)^{\frac{1}{r}}\\
			&\quad +
				\left(
				\sum_{k\in\K}
				\left(\int_{\Delta_{k-1}}
				\vp(z)^{\frac{mp}{p-m}-1}\vp'(z)
				\left(\int_{z}^{\infty}v(s)^{1-p'}\ds\right)^{\!\!\frac{m(p-1)}{p-m}}\dz
				\right)^{\!\!\frac{q(p-m)}{m(p-q)}}
				\right)^{\!\!\frac {1}{r}}\\
			&\lesssim
			\left(
			\sum_{k\in\K}
			\vp(t_{k-1})^r
			\left(\int_{t_{k-1}}^{\infty}v(s)^{1-p'}\ds\right)^{\frac{r}{p'}}\right)^{\frac{1}{r}}\\
			&\quad +
			\left(
			\sum_{k\in\K}
			\left(\int_{\Delta_{k-1}}
			\vp(z)^{\frac{mp}{p-m}-1}\vp'(z)
			\left(\int_{z}^{t_k}v(s)^{1-p'}\ds\right)^{\!\!\frac{m(p-1)}{p-m}}\!\!\!\dz
			\right)^{\!\!\frac{q(p-m)}{m(p-q)}}
			\right)^{\!\!\frac {1}{r}}\\
			&\lesssim
				\left(
				\sum_{k\in\K}
				\vp(t_{k-1})^r
				\left(\int_{\Delta_{k-1}} v(s)^{1-p'}\ds\right)^{\frac{r}{p'}}\right)^{\frac{1}{r}} + D_4\\
			&\lesssim D_4.
	\end{align*}
As for $\B{17}$, by \eqref{E:copson-6} and Lemma \ref{L:20} one has
	\[
	\B{17} \le \left( \sum_{k\in\K} \vp(t_{k-1})^r \left(\int_{\Delta_{k-1}} v(s)^{1-p'}\ds\right)^{\frac{r}{p'}} \right)^{\frac{1}{r}} \lesssim D_4.
	\]
We proceed with the term $\B{18}$. Integrating by parts again, as well as using \eqref{E:copson-5} and Lemma \ref{L:20}, we obtain the following:
\Bdef{19}\Bdef{20}
\begin{align*}
\B{18}
&\lesssim
\left(\vphantom{\Bigg[^\big[}
\sum_{k\in\K\setminus\{K\}}
\left(\int_0^{t_k} w(s)\ds\right)^{\frac{p-m}{p-q}+\frac{m}{p-q}}\right.\\
&\qquad\times\left.\!
\left(\int_{t_k}^{\infty}
\left(\int_{t_k}^{z}u(s)\ds
\right)^{\!\!\frac{m}{p-m}}
\!u(z)
\left(\int_{z}^{\infty}\!\!v(s)^{1-p'}\ds\right)^{\!\!\frac{m(p-1)}{p-m}}\!\!\dz
\right)^{\!\!\!\frac{q(p-m)}{m(p-q)}}
\right)^{\!\!\frac {1}{r}}\\
&\lesssim
\left(\vphantom{\Bigg[^\big[}
\sum_{k\in\K\setminus\{K\}}\!\!
\left(\int_0^{t_k} \!\!w(s)\ds\right)^{\!\!\frac{p-m}{p-q}} \right.\\
&\left.\qquad\times
\left(\int_{t_k}^{\infty}
\vp(z)^{\frac{m^{2}}{p-m}}u(z)
\left(\int_{z}^{\infty}\!\!v(s)^{1-p'}\ds\right)^{\!\!\frac{m(p-1)}{p-m}}\!\!\!\!\dz
\right)^{\!\!\frac{q(p-m)}{m(p-q)}}
\right)^{\!\!\frac {1}{r}}\\
&\lesssim
\left(\vphantom{\Bigg[^\big[}
\sum_{k\in\K\setminus\{K\}}
\left(\int_0^{t_k} w(s)\ds\right)^{\frac{p-m}{p-q}}\right.\\
&\left.\qquad\times
\left(\int_{\Delta_k}
\vp(z)^{\frac{m^{2}}{p-m}}u(z)
\left(\int_{z}^{\infty}v(s)^{1-p'}\ds\right)^{\frac{m(p-1)}{p-m}}\dz
\right)^{\frac{q(p-m)}{m(p-q)}}
\right)^{\frac {1}{r}}\\
&\lesssim
\left(\vphantom{\Bigg[^\big[}
\sum_{k\in\K\setminus\{K\}}
\left(\int_0^{t_k} w(s)\ds\right)^{\frac{p-m}{p-q}}\right.\\
&\left.\qquad\times
\left(\int_{\Delta_k}
\vp(z)^{\frac{m^{2}}{p-m}}u(z)
\left(\int_{z}^{t_{k+1}}v(s)^{1-p'}\ds\right)^{\frac{m(p-1)}{p-m}}\dz
\right)^{\frac{q(p-m)}{m(p-q)}}
\right)^{\frac {1}{r}}\\
&\quad +
\left(\vphantom{\Bigg[^\big[}
\sum_{k\in\K\setminus\{K,K-1\}}
\left(\int_0^{t_k} w(s)\ds\right)^{\frac{p-m}{p-q}}\right.\\
&\left.\qquad\times
\left(\int_{\Delta_k}
\vp(z)^{\frac{m^{2}}{p-m}}u(z) \dz
\right)^{\frac{q(p-m)}{m(p-q)}}
\left(\int_{t_{k+1}}^{\infty}v(s)^{1-p'}\ds\right)^{\frac{r}{p'}}
\right)^{\frac {1}{r}}\\
&=\B{19}+\B{20},
\end{align*}
say. To deal with $\B{19}$ we recall the inequality \eqref{E:Hardy-lemma} which is valid here as well, since the relation between $p$ and $m$ is the same as in the case (c). Hence, the Hardy lemma gives
	\begin{align*}
		\B{19} &\lesssim \left(
		\sum_{k\in\K\setminus\{K\}}
		\left(\int_0^{t_k} \!\!w(s)\ds\right)^{\!\!\frac{p-m}{p-q}}\right.\\
		&\left.\quad\times \!\left(\int_{\Delta_k}
		\!\vp(z)^{\frac{m^{2}}{p-m}} \int_{t_k}^z \!u(y)\dy
		\left(\int_{z}^{t_{k+1}} \hspace{-5pt} v(s)^{1-p'}\ds\right)^{\!\!\frac{p(m-1)}{p-m}} \hspace{-12pt}  v(z)^{1-p'}\!\dz
		\right)^{\!\!\frac{q(p-m)}{m(p-q)}}
		\right)^{\!\!\frac {1}{r}}\\
		& \lesssim D_4,
	\end{align*}
since $\frac{m^2}{p-m}+m=\frac{mp}{p-m}$. In addition to this, the standard argument involving \eqref{E:copson-6} and Lemma \ref{L:20} shows that
	\begin{align*}
		\B{20} & \!\lesssim \!\!\left(\!
		\sum_{k\in\K\setminus\{K\!,K\!-\!1\}}\!\!\!\!
		\left(\int_0^{t_k} \hspace{-10pt}w(s)\!\ds\!\right)^{\hspace{-5pt}\frac{p-m}{p-q}}
		 \hspace{-6pt}\left(\int_{\!\Delta_k}
		\hspace{-6pt}u(z)\!\dz\!
		\right)^{\hspace{-4pt}\frac{q(p-m)}{m(p-q)}} \hspace{-12pt}\vp(t_{k+1})^{\frac{mq}{p-q}}
		\hspace{-4pt}\left(\int_{t_{k+1}}^{\infty}\hspace{-10pt}v(s)^{1-p'}\!\!\ds\!\right)^{\hspace{-4pt}\frac{r}{p'}}
		\!\right)^{\hspace{-5pt}\frac {1}{r}}\\
		& \lesssim \left(
		\sum_{k\in\K\setminus\{K,K-1\}}
				 \vp(t_{k+1})^{r}
		\left(\int_{t_{k+1}}^{\infty}v(s)^{1-p'}\ds\right)^{\frac{r}{p'}}
		\right)^{\frac {1}{r}},
		\end{align*}
which implies $ \B{20}\lesssim D_4$. Altogether, we have proved
	\[
	 A_5 \lesssim D_4.
	\]
What remains is to find an~estimate for $D_4$. By Lemma \ref{L:40}, there exists a~non-negative sequence $\{c_k\}_{k\in\K}$ satisfying $\sum c_k^p=1$ and such that
\begin{align*}
D_4 &=
\left(
\sum_{k\in\K}
c_k^q
\left(\int_{\Delta_{k-1}}
\vp(t)^{\frac{mp}{p-m}}
\left(\int_{t}^{t_k}v(s)^{1-p'}\ds\right)^{\!\!\frac{p(m-1)}{p-m}}
v(t)^{1-p'}\dt\right)^{\!\!\frac{q(p-m)}{pm}}\right)^{\!\!\frac{1}{q}}\!\!.
\end{align*}
For all $t>0$, let us define
	\[
		g(t)= \sum_{k\in\K} c_k g_k(t) \qquad \text{and} \qquad f(t)= \sum_{k\in\K} c_k f_k(t),
	\]
where the functions $g_k$ and $f_k$ are defined earlier (see \eqref{E:nutnost-2} and \eqref{E:nutnost-f}). Then, integrating by parts, using the saturation properties of $g_k$ and $f_k$ and recalling \eqref{E:diskr}, we get
\begin{align*}
D_4\!
&\approx\!\!
\left(
\sum_{k\in\K}
c_k^q
\vp(t_{k-1})^{q}
\left(\int_{\Delta_{k-1}}
v(s)^{1-p'}\ds\right)^{\frac{q}{p'}}\right)^{\frac{1}{q}}\\
&\quad +\hspace{-4pt}
\left(
\!\sum_{k\in\K}\!
c_k^q\hspace{-4pt}
\left(\!\int_{\Delta_{k\!-\!1}}\hspace{-6pt}
\left(\!\int_{t_{k-1}}^{t}\hspace{-13pt}\vp(s)^{m\!-\!1}\vp'(s)\!\ds\!\right)^{\hspace{-4pt}\frac{p}{p-m}}
\hspace{-7pt}\left(\int_{t}^{t_k}\hspace{-8pt}v(s)^{1-p'}\!\ds\!\right)^{\hspace{-5pt}\frac{p(m\!-\!1)}{p-m}}
\hspace{-18pt}v(t)^{1-p'} \!\!\dt\!\right)^{\hspace{-7pt}\frac{q(p-m)}{pm}}\right)^{\hspace{-6pt}\frac{1}{q}}\\
&\lesssim
\left(
\sum_{k\in\K}
c_k^q
\vp(t_{k-1})^{q}
\left(\int_{\Delta_{k-1}}
g_k(s)\ds\right)^{q}\right)^{\frac{1}{q}}\\
&\quad +
\left(
\sum_{k\in\K}
c_k^q
\left(\int_{\Delta_{k-1}}
\vp(y)^{m-1}\vp'(y)
\left(\int_{y}^{t_k}f_k(s)\ds\right)^{m}\dy\right)^{\frac{q}{m}}\right)^{\frac{1}{q}},
\end{align*}
and, consequently,
\begin{align*}
D_4
&\lesssim
\left(
\sum_{k\in\K}
\vp(t_{k-1})^{q}
\left(\int_{\Delta_{k-1}}
(g(s)+f(s))\ds\right)^{q}\right)^{\frac{1}{q}}\\
&\quad +
\left(
\sum_{k\in\K}
\left(\int_{\Delta_{k-1}}
\vp(y)^{m-1}\vp'(y)
\left(\int_{y}^{t_k}(g(s)+f(s))\ds\right)^{m}\dy\right)^{\frac{q}{m}}\right)^{\frac{1}{q}}\\
&\approx
\left(
\int_0^{\infty}
w(t)\left(\int_{t}^{\infty}u(s)\left(\int_{s}^{\infty}(g(y)+f(y))\dy\right)^m\ds\right)^{\frac{q}{m}}\dt\right)^{\frac{1}{q}} \\
&\le C\|g+f\|_{L^{p}(v)}.
\end{align*}
Since $\|f+g\|_{L^p(v)} \le 2$, this shows that
	\[
		D_4 \lesssim C,
	\]
and the proof is complete.
\end{proof}

\begin{proof}[Proof of Proposition \ref{T:dodatek}]
	For simplicity, we assume that $\int_{t}^{\infty}\!v(s)^{1-p'}\!\ds<\infty$ for every $t\in(0,\infty)$ (otherwise, the proof needs only a~slight modification in the same spirit as in the proof of Theorem \ref{T:main}). Then it is easy to observe that $A_1\lesssim A_6$. Thus, we need to prove that $C \lesssim A_5+A_6$, and that the reverse inequality holds if and only if $q>m$.
	
	Theorem \ref{T:main}(d) yields $C\approx A_3+A_5$. Hence, to prove that $C \lesssim A_5+A_6$ it suffices to show that $A_3\lesssim A_6$. Changing the order of integration, one has
		\begin{align*}
			\int_0^t w(s) \left( \int_s^t u(y) \dy \right)^\frac qm \ds & \approx \int_0^t w(s) \int_s^t \left( \int_s^z u(y)\dy \right)^{\frac qm -1} u(z) \dz \ds \\
																		& = \int_0^t \int_0^z w(s) \left( \int_s^z u(y)\dy \right)^{\frac qm -1} \ds\  u(z)  \dz.
		\end{align*}	
	Using this observation and partial integration, we get the following:
		\begin{align*}
			A_3^r  & =        \int_0^{\infty} \left(\int_0^tw(s)\left(\int_s^t u(y)\dy\right)^{\frac qm}\ds\right)^{\frac rp} w(t) \\
					& \qquad \times \sup_{z\in(t,\infty)} \left(\int_t^z u(y)\dy\right)^{\frac qm} \left(\int_z^{\infty}v(y)^{1-p'}\dy\right)^{\frac r{p'}}\dt \\
				   & \lesssim \int_0^{\infty} \left(\int_0^tw(s)\left(\int_s^t u(y)\dy\right)^{\frac qm}\ds\right)^{\frac rp} w(t) \\
				   & \qquad \times \int_t^\infty \left(\int_t^z u(y)\dy\right)^{\frac qm -1} u(z) \left(\int_z^{\infty}v(y)^{1-p'}\dy\right)^{\frac r{p'}} \dz \dt \\
				   & \lesssim \int_0^{\infty} \int_0^z \left(\int_0^tw(s)\left(\int_s^t u(y)\dy\right)^{\frac qm}\ds\right)^{\frac rp} w(t) \\
				   & \qquad \times \left(\int_t^z u(y)\dy\right)^{\frac qm -1} \dt \ u(z) \left(\int_z^{\infty}v(y)^{1-p'}\dy\right)^{\frac r{p'}} \dz \\
				   & \le      \int_0^{\infty} \left(\int_0^z w(s)\left(\int_s^z u(y)\dy\right)^{\frac qm}\ds\right)^{\frac rp} \\
				   &\qquad \times \int_0^z w(t) \left(\int_t^z u(y)\dy\right)^{\frac qm -1} \dt \ u(z) \left(\int_z^{\infty}v(y)^{1-p'}\dy\right)^{\frac r{p'}} \dz \\
				   & \lesssim \int_0^{\infty} \left(\int_0^z w(s)\left(\int_s^z u(y)\dy\right)^{\frac qm}\ds\right)^{\frac rq} \left(\int_z^{\infty}v(y)^{1-p'}\dy\right)^{\frac r{q'}} v^{1-p'}(z) \dz \\
				   & = A_6^r.
		\end{align*}	
	We have shown that
		$$ C \lesssim A_5+A_6 $$
	independently of any relation between $q$ and $m$.	
	
	As next, we want to show the opposite estimate for $q>m$. It may be done using the discretization from the proof of Theorem \ref{T:main}. However, the fact that $\frac qm>1$ allows to use the following $L^p$-duality argument which is simpler. We have
		\begin{align*}
			C & = \sup_{h\in\M_+}
					\frac
					{\left(\int_0^{\infty}\left(\int_t^\infty \left(\int_s^{\infty}h(y)\dy\right)^mu(s) \ds\right)^{\frac{q}{m}}w(t)\dt\right)^{\frac{m}{q}\cdot \frac 1m}}
					{\left(\int_0^{\infty}h(t)^pv(t)\dt\right)^{\frac{1}{p}}}\\
			  & = \sup_{h\in\M_+} \sup_{g\in\M_+}
				  	\frac
				  	{\left(\int_0^{\infty} \int_t^\infty \left(\int_s^{\infty}h(y)\dy\right)^mu(s) \ds\ g(t)\dt\right)^{\frac{1}{m}}}
				  	{\left(\int_0^{\infty}h(t)^p v(t)\dt\right)^{\frac{1}{p}} \left( \int_0^\infty g(t)^\frac{q}{q-m} w(t)^{-\frac{m}{q-m}} \dt \right)^\frac{q-m}{mq}}\\
			  & = \sup_{g\in\M_+} \frac 1{\left( \int_0^\infty g(t)^\frac{q}{q-m} w(t)^{-\frac{m}{q-m}} \dt \right)^\frac{q-m}{mq}} \\
			  & \qquad\qquad \times \sup_{h\in\M_+}
					  \frac
					  {\left(\int_0^{\infty} \left(\int_s^{\infty}h(y)\dy\right)^mu(s) \int_0^s g(t)\dt \ds\right)^{\frac{1}{m}}}
					  {\left(\int_0^{\infty}h(t)^p v(t)\dt\right)^{\frac{1}{p}} }\\
			  & \approx \sup_{g\in\M_+} \frac
				   {\!\left( \int_0^\infty \!\left( \int_0^t g(s) \int_s^t u(y) \dy \right)^\frac{p}{p-m} \!\left( \int_t^\infty v(y)^{1-p'} \dy \right)^\frac{p(m-1)}{p-m} \!\!v(t)^{1-p'} \!\dt \right)^{\!\!\frac{m-p}{mp}}}
				   {\left( \int_0^\infty g(t)^\frac{q}{q-m} w(t)^{-\frac{m}{q-m}} \dt \right)^\frac{q-m}{mq}}  \\
			  & \approx A_5 + A_6.
		\end{align*}	
	In the last two steps one uses the Hardy inequality (Lemma  \ref{L:35}(ii))	and the characterization of boundedness of a~Hardy operator with a~kernel from \cite[Theorem 1.2]{Oi}. Notice that instead of \cite{Oi}	one may also use the alternative characterization from \cite[Theorem 8]{K-complut} which gives $C \approx A_3 + A_5$, yet still only under the condition $q>m$ which is needed for the duality argument.
	
	If $m=q$, the equivalence $C \approx A_6$ is obtained by a~direct use of the Hardy inequality, and it is easily verified that $A_5\le A_6$. We omit the details. 													

	Now assume that $q<m$. We need to show that the inequality $A_6 \le \beta C$ is not satisfied for any positive constant $\beta$ independent of the involved weights. For $t>0$ and $n\in\N$, let us define
		$$
			u(t)\!=\!1, \quad v(t) \!= \!\chi_{[0,\frac12]}(t)  t^{\frac pm +p- 1} |\log t|^{\frac{p-q}{p'(q-1)}} + \chi_{(\frac12,\infty)}(t) e^t\!, \quad w_n(t)\!=\! n \chi_{[0,\frac1n]}(t).
		$$
	By monotonicity, for any $h\in\MM$ one has
		$$
			\left(\int_0^{\infty}\!\!\left(\int_t^\infty \!\!\left(\int_s^{\infty}\!\!\!h(y)\dy\!\right)^{\!m} \!\! u(s) \ds\!\right)^{\!\!\frac{q}{m}}\!\!\!w_n(t)\dt\!\right)^{\!\!\frac{1}{q}} \!\!\le\! \left(\int_0^{\infty} \!\!\left(\int_s^{\infty}\!\!\!h(y)\dy\right)^{\!m} \!\!\!\!u(s) \ds\!\right)^{\!\!\frac{1}{m}}\!\!.
		$$	
Hence, one has 
		\begin{align*}
		C_n & = \sup_{h\in\M_+}
		\frac
		{\left(\int_0^{\infty}\left(\int_t^\infty \left(\int_s^{\infty}h(y)\dy\right)^mu(s) \ds\right)^{\frac{q}{m}}w_n(t)\dt\right)^{\frac{1}{q}}}
		{\left(\int_0^{\infty}h(t)^pv(t)\dt\right)^{\frac{1}{p}}}\\
		&\le \sup_{h\in\M_+}
		\frac
		{ \left(\int_0^{\infty} \left(\int_s^{\infty}h(y)\dy\right)^mu(s) \ds\right)^{\frac{1}{m}} }
		{\left(\int_0^{\infty}h(t)^pv(t)\dt\right)^{\frac{1}{p}}}.
    \end{align*}
Using Lemma \ref{L:35}(ii) and a~careful computation of the exponents, we get
\begin{equation*}
		C_n \lesssim \left( \int_0^\infty \left( \int_0^t u(s)\ds \right)^\frac{p}{p-m} \left(\int_t^\infty v(s)^{1-p'} \ds \right)^{\!\!\frac{p(m-1)}{p-m}} \!\!v(t)^{1-p'} \dt \right)^{\!\!\!\!\frac{p-m}{mp}}\!\!\!\! < \infty.
\end{equation*}
On the other hand, at each point $t\in(0,\infty)$ we have
		$$
			0 \le \left(\int_0^tw_n(s)
			\left(\int_s^{t}u(y)\dy\right)^{\frac q{m}}\ds \right)^\frac rq \big\uparrow \left( \int_0^t u(y)\dy \right)^\frac rm,
		$$
		thus the monotone convergence theorem yields
		\begin{align*}
			A_{6,n} & =\! \left(\int_0^{\infty}\!\!
						\left(\int_0^tw_n(s)
						\left(\int_s^{t}u(y)\dy\right)^{\frac q{m}}\!\!\ds\right)^{\!\!\frac rq}\!\!
						\left(\int_t^{\infty}\!\!\!v(y)^{1-p'}\dy\right)^{\!\!\frac{r}{q'}}\!\!v(t)^{1-p'}\dt\right)^{\!\!\frac{1}{r}} \\
					& \qquad \xrightarrow{n\to\infty}
						\left(\int_0^{\infty}
						\left(\int_0^{t}u(y)\dy\right)^{\frac r{m}}
						\left(\int_t^{\infty}v(y)^{1-p'}\dy\right)^{\frac{r}{q'}}v(t)^{1-p'}\dt\right)^{\frac{1}{r}}.
		\end{align*}
By an~analysis of the exponents at the power and logarithmic functions, one may check that the limit term is infinite. Hence, we have shown that, for any given (but fixed) parameters $0<q<m<p$, $1<p<\infty$, there exist weights $u$, $v$, and a~sequence of weights $\{w_n\}_{n\in\N}$ such that the values of $C_n$ are uniformly bounded but the values of $A_{6,n}$ diverge to infinity. Thus, the inequality $A_6 \le \beta C$ does not hold with any $0<\beta<\infty$ independent of the involved weights. 	
\end{proof}

The proof of Theorem~\ref{T:main-p=1} needs only rather obvious changes compared to that of Theorem~\ref{T:main} and is therefore omitted. The assertion of Corollary~\ref{Cor} follows directly from the proof of Theorem~\ref{T:main}.

\subsection*{Acknowledgements}
We would like to thank both the referees for a thorough and critical reading of the manuscript including the highly technical parts and for bringing to our attention many valuable and useful comments. 

This research was supported by the grant P201/13/14743S and\linebreak P201/18/00580S of the Grant Agency of
the Czech Republic and by the Danube Region Grant no. 8X17028.

\end{document}